\documentclass[11pt,a4paper, tikz, british]{amsart}

\usepackage{babel}
\usepackage[useregional]{datetime2}
\DTMlangsetup[en-GB]{en-GB}

\usepackage{amssymb}
\usepackage[hidelinks]{hyperref}

\usepackage{tikz}
\usepackage{pgfmath}
\usepackage{caption}
\usetikzlibrary{arrows}

\usepackage{xcolor}
\selectcolormodel{gray}

\pgfmathsetmacro{\myxlow}{-2}
\pgfmathsetmacro{\myxhigh}{2}
\pgfmathsetmacro{\myiterations}{3}

\usepackage{pgfplots}
\usepackage{caption}
\pgfdeclarelayer{background}
\pgfsetlayers{background,main}

\newtheorem{thm}{Theorem}[section]
\newtheorem{prop}[thm]{Proposition}
\newtheorem{lem}[thm]{Lemma}
\newtheorem{cor}[thm]{Corollary}
\newtheorem{conj}[thm]{Conjecture} 

\theoremstyle{definition}
\newtheorem{definition}[thm]{Definition}
\newtheorem{example}[thm]{Example}

\theoremstyle{remark}

\numberwithin{equation}{section}

\newcommand{\G}{\mathbb{G}_\mathrm{m}}  % The multiplicative group.
\newcommand{\alg}{\overline{\mathbb{Q}}} %The algebraic numbers
\newcommand{\R}{\mathbb{R}}  % The real numbers.
\newcommand{\Q}{\mathbb{Q}} % The rational numbers.
\newcommand{\Z}{\mathbb{Z}} % The integers
\newcommand{\C}{\mathbb{C}} % The complex numbers.
\newcommand{\h}{\mathbb{H}} % The upper half plane
\newcommand{\GL}{\mathrm{GL}_2^+(\Q)} % Rational matrices with positive determinant
\newcommand{\MZ}{\mathrm{M}_2(\Z)} % Integer 2x2 matrices
\newcommand{\SL}{\mathrm{SL}_2(\Z)} % Modular group
\newcommand{\RAE}{\R_{\mathrm{an},\exp}} %R_an,exp
\DeclareMathOperator{\re}{Re} % Real part
\DeclareMathOperator{\im}{Im} % Imaginary part
\begin{document}
	
	\title[Products of differences of singular moduli]{Multiplicative relations among differences of singular moduli}
	
	\author{Vahagn Aslanyan}
	\address{School of Mathematics, University of Leeds, Leeds, LS2 9JT, UK}
	\curraddr{Department of Mathematics, University of Manchester, Manchester, M13 9PL, UK}
	\email{\href{Vahagn.Aslanyan@manchester.ac.uk}{Vahagn.Aslanyan@manchester.ac.uk}}
	\author{Sebastian Eterovi\'{c}}
	\address{School of Mathematics, University of Leeds, Leeds, LS2 9JT, UK}	
	\curraddr{Kurt G\"odel Research Center, Universit\"at Wien, 1090 Wien, Austria}
	\email{\href{sebastian.eterovic@univie.ac.at}{sebastian.eterovic@univie.ac.at}}
	
	\author{Guy Fowler}
	\address{Institut f\"{u}r Algebra, Zahlentheorie und Diskrete Mathematik,\newline  Leibniz Universität Hannover, 30167 Hannover, Germany}
	\curraddr{Department of Mathematics, University of Manchester, Manchester, M13 9PL, UK}
	\email{\href{guy.fowler@manchester.ac.uk}{guy.fowler@manchester.ac.uk}}
	\date{\today}

	\thanks{VA was supported by Leverhulme Trust Early Career Fellowship ECF-2022-082 at the University of Leeds (where this work was done), and by EPSRC Open Fellowship EP/X009823/1 at the University of Manchester. SE was supported by EPSRC fellowship EP/T018461/1. GF has received funding from the European Research Council (ERC) under the European Union’s Horizon 2020 research and innovation programme (grant agreement no. 945714).}

	\begin{abstract}
		Let $n \in \Z_{>0}$. We prove that there exist a finite set $V$ and finitely many algebraic curves $T_1, \ldots, T_k$ with the following property: if $(x_1, \ldots, x_n, y)$ is an $(n+1)$-tuple of pairwise distinct singular moduli such that $\prod_{i=1}^n (x_i - y)^{a_i}=1$ for some $a_1, \ldots, a_n \in \Z \setminus \{0\}$, then $(x_1, \ldots, x_n, y) \in V \cup T_1 \cup \ldots \cup T_k$. Further, the curves $T_1, \ldots, T_k$ may be determined explicitly for a given $n$.  
	\end{abstract}
	
	\dedicatory{To Jonathan Pila}
	
	\maketitle
	
	\section{Introduction}
	
	Let $\h$ denote the complex upper half plane. The modular group $\SL$ acts on $\h$ by fractional linear transformations. The modular $j$-function $j \colon \h \to \C$ is the unique holomorphic function $\h \to \C$ which is invariant under this action of $\SL$, has a simple pole at $i \infty$, and satisfies $j(i)=1728$ and $j(\rho) = 0$, where $\rho = \exp(2 \pi i /3)$. 
	
	A singular modulus is a complex number $j(\tau)$ for some $\tau \in \h$ such that $[\Q(\tau) : \Q] = 2$. For example, $0$ and $1728$ are both singular moduli. Equivalently, a singular modulus is the $j$-invariant of an elliptic curve with complex multiplication. Singular moduli are algebraic integers and generate the ring class fields of imaginary quadratic fields. By Schneider's theorem \cite[IIc]{Schneider37}, if $\tau \in \h$ is such that both $\tau, j(\tau) \in \alg$, then $j(\tau)$ is a singular modulus.
	
	In this paper, we consider multiplicative relations among differences $x-y$ of singular moduli $x, y$. Since $0$ is a singular modulus, every singular modulus is equal to the difference of two singular moduli. Our aim is to generalise the following theorem. 
	
	\begin{thm}\label{thm:fixedy}
		Let $n \in \Z_{>0}$. Let $y$ be a singular modulus. Then there exist only finitely many $n$-tuples $(x_1, \ldots, x_n)$ of pairwise distinct singular moduli $x_1, \ldots, x_n$ such that $y \notin \{x_1, \ldots, x_n\}$ and there exist $a_1, \ldots, a_n \in \Z \setminus \{0\}$ for which
		\[\prod_{i=1}^n (x_i - y)^{a_i} = 1.\]
	\end{thm}
	
	Theorem~\ref{thm:fixedy} was proved by Pila and Tsimerman \cite{PilaTsimerman17} for $y=0$ and by the third author \cite{Fowler21} for $y$ not in the real interval $(0, 1728)$. In Section~\ref{sec:fixedy}, we show that the result of \cite{Fowler21} directly implies the remaining case where $y$ is in the real interval $(0, 1728)$.
	
	This paper addresses the case where $y$ is allowed to vary over all singular moduli. That is, we consider $(n+1)$-tuples $(x_1, \ldots, x_n, y)$ of pairwise distinct singular moduli $x_1, \ldots, x_n, y$ such that
	\begin{align}\label{eq:dep}
		\prod_{i=1}^n (x_i - y)^{a_i} =1 \mbox{ for some $a_1,\ldots, a_n \in \Z \setminus \{0\}$}.
	\end{align}
	In this setting, one must account for the following situation. 
	
	\begin{definition}[{\cite[p.~1052]{BiluLucaMasser17}}]\label{def:jmap}
		A function $f \colon \h \to \C$ is called a $j$-map if either there exists a singular modulus $x$ such that $f(z) = x$ for every $z \in \h$, or there exists $g \in \GL$ such that $f(z) = j(g z)$ for every $z \in \h$. Here $\GL$ acts on $\h$ by fractional linear transformations.
	\end{definition}
	
	\begin{definition}
		Let $n \in \Z_{>0}$. Let $f_1, \ldots, f_n, f$ be pairwise distinct $j$-maps, at least one of which is non-constant. The set
		\[ \Big\{(f_1(z), \ldots, f_n(z), f(z)) : z \in \h\Big\}\]
		is called a multiplicative special curve in $\C^{n+1}$ if there exist $a_1, \ldots, a_n \in \Z \setminus \{0\}$ such that, for all $z \in \h$,
		\[ \prod_{i=1}^n (f_i(z) - f(z))^{a_i} = 1.\]
	\end{definition}
	
	Note that a multiplicative special curve is always an algebraic curve (see Proposition~\ref{prop:jset}). Clearly, any multiplicative special curve contains infinitely many $(n+1)$-tuples $(x_1, \ldots, x_n, y)$ of pairwise distinct singular moduli satisfying \eqref{eq:dep}. If $N \in \Z_{>0}$ is not a perfect square, then the modular polynomial $\Phi_N \in \Z[X, Y]$ gives rise to a multiplicative special curve, as we explain in Section~\ref{subsec:families}. Thus one cannot hope to show, for an arbitrary $n \in \Z_{>0}$, that there exist only finitely many such $(n+1)$-tuples $(x_1, \ldots, x_n, y)$.
	
	Instead, we prove that the multiplicative special curves arising from the modular polynomials are the only multiplicative special curves. In particular, for a given $n$, there are only finitely many multiplicative special curves in $\C^{n+1}$ and these may be determined effectively. 
	
	\begin{thm}\label{thm:MSC}
		Let $n \in \Z_{>0}$. Then there are only finitely many multiplicative special curves in $\C^{n+1}$ and these may be determined effectively. If $n \leq 5$, then there are no multiplicative special curves in $\C^{n+1}$.
	\end{thm}
	
	We then prove that, for every $n \in \Z_{>0}$, the finitely many multiplicative special curves in $\C^{n+1}$ account for all but finitely many of the $(n+1)$-tuples $(x_1, \ldots, x_n, y)$ of pairwise distinct singular moduli satisfying \eqref{eq:dep}. 
	
	\begin{thm}\label{thm:main}
		Let $n \in \Z_{>0}$. Then there exist only finitely many $(n+1)$-tuples $(x_1, \ldots, x_n, y)$ of pairwise distinct singular moduli $x_1, \ldots, x_n, y$ such that
		\[\prod_{i=1}^n (x_i - y)^{a_i} = 1\]
		for some $a_1, \ldots, a_n \in \Z \setminus \{0\}$ and $(x_1, \ldots, x_n, y)$ does not belong to one of the finitely many multiplicative special curves in $\C^{n+1}$. 
	\end{thm}
	
	Since there are no multiplicative special curves in $\C^{n+1}$ for $n \leq 5$, one immediately obtains the following corollary. Since (see Example~\ref{eg:mod}) there exists a multiplicative special curve in $\C^7$, the bound of $n \leq 5$ in this corollary is sharp.
	
	\begin{cor}\label{cor:small}
		Let $n \in \{1, \ldots, 5\}$. There exist only finitely many $(n+1)$-tuples $(x_1, \ldots, x_n, y)$ of pairwise distinct singular moduli $x_1, \ldots, x_n, y$ such that
		\[\prod_{i=1}^n (x_i - y)^{a_i} = 1\]
		for some $a_1, \ldots, a_n \in \Z \setminus \{0\}$.
	\end{cor}

	The proof of Theorem~\ref{thm:main} uses o-minimality and is ineffective. Recently, Li \cite{Li21} has proved that the difference of two singular moduli is never a unit (in the ring of algebraic integers). Hence, there are no distinct singular moduli $x, y$ such that $(x-y)^a =1$ for some $a \in \Z \setminus \{0\}$. 
	
	\subsection{Modular polynomials and multiplicative special curves}\label{subsec:families}
	
	For background on modular polynomials, see \cite[\S11]{Cox89}. For $N \in \Z_{>0}$, let 
	\[ C(N) = \Big \{ \begin{pmatrix}
		a & b\\
		0 & d
	\end{pmatrix} \in \MZ : ad=N, a>0, 0 \leq b < d, \gcd(a, b, d) = 1 \Big \}.\] 
	There exists \cite[(11.15)]{Cox89} a polynomial $\Phi_N \in \Z[X, Y]$ with the property that
	\[ \Phi_N(X, j(z)) = \prod_{g \in C(N)} (X - j(g z)) \]
	for all $z \in \h$. The polynomial $\Phi_N$ is called the $N$th modular polynomial. 
	
	For $N > 1$, let $F_N \in \Z[X]$ be defined by $F_N(X) = \Phi_N(X, X)$. Then $F_N$ is a non-constant polynomial (the explicit formula in \cite[Proposition~13.8]{Cox89} in fact implies that $\deg F_N \geq 2N$). The roots of $F_N$ are all singular moduli (see Corollary~\ref{cor:rootsofF}). If $N$ is not a perfect square, then, by \cite[Theorem~11.18]{Cox89}, the polynomial $F_N$ has leading coefficient $\pm 1$. 
	
	Suppose then that $N \in \Z_{>1}$ is such that the leading coefficient of $F_N$ is $\pm 1$ (e.g.~take $N$ not a perfect square). Write $\alpha_1, \ldots, \alpha_k$ for the distinct roots of $F_N$ and $a_i$ for their multiplicities. Write $g_1, \ldots, g_l$ for the elements of $C(N)$. Since
	\[ F_N(j(z)) = \prod_{i=1}^l (j(z) - j(g_i z)) \]
	for all $z \in \h$, one thus obtains (doubling the exponents to eliminate a potential factor of $-1$) that 
	\begin{align*}
		\prod_{i=1}^k (j(z) - \alpha_i)^{2 a_i} = \prod_{i=1}^l (j(z) - j(g_i z))^2
	\end{align*}
	for all $z \in \h$, and hence, for all $z \in \h$,
	\begin{align}\label{eq:triv}
		\prod_{i=1}^k (j(z) - \alpha_i)^{2 a_i} \prod_{i=1}^l (j(z) - j(g_i z))^{-2}= 1.
	\end{align}
	In particular, the set
	\[ \Big\{(\alpha_1, \ldots, \alpha_k, j(g_1 z), \ldots, , j(g_l z), j(z)) : z \in \h\Big\}\]
	is a multiplicative special curve in $\C^{k+l+1}$.
	
	Further examples of multiplicative special curves may be generated by multiplying together integer powers of relations of the form \eqref{eq:triv} coming from different $F_N$. In this case, one must also consider polynomials $F_N$ with leading coefficient not equal to $\pm 1$, because these leading coefficients may cancel with one another. For example, $-2$ is the leading coefficient of both $F_4$ and $F_{16}$. Theorem~\ref{thm:MSCshape} will show that all the multiplicative special curves arise from the polynomials $F_N$ in this way.

	\begin{example}\label{eg:mod}
		The modular polynomial $\Phi_2$ is
		\begin{align*}
			\Phi_2(X, Y) =& -X^2Y^2 + X^3 + Y^3 + 1488(X^2Y + XY^2)\\ &-162000(X^2+Y^2) +40773375 XY \\
			& +8748000000(X+Y) - 157464000000000.
		\end{align*}
		Thus,
		\begin{align*}
			F_2(X) = &\Phi_2(X, X)\\
			= &-X^4 + 2978 X^3 + 40449375 X^2 + 17496000000 X\\ 
			&- 157464000000000\\
			=& - (X-1728)(X+3375)^2(X-8000).
		\end{align*}
		Note that $1728$, $-3375$, $8000$ are singular moduli of discriminant $-4$, $-7$, $-8$ respectively. The discriminant of a singular modulus $j(\tau)$ is $b^2-4ac$, where $a, b, c \in \Z$, not all zero, are such that $a \tau^2 +b \tau +c = 0$ and $\gcd(a,b,c)=1$. 
		
		Observe that
		\[C(2) = \Big \{\begin{pmatrix}
			2 & 0\\
			0 & 1
		\end{pmatrix}, 
		\begin{pmatrix}
			1 & 0\\
			0 & 2
		\end{pmatrix}, 
		\begin{pmatrix}
			1 & 1\\
			0 & 2
		\end{pmatrix}\Big \}.\]
		Thus,
		\[\Phi_2(X, j(z)) = \Big(X - j(2z)\Big )\Big(X - j\Big (\frac{z}{2}\Big) \Big)\Big(X - j\Big(\frac{z+1}{2}\Big)\Big).\]
		Hence, for all $z \in \h$,
		\begin{align}\label{eq:phi2}
			\begin{split}
				&-(j(z)-1728)(j(z)+3375)^2(j(z) - 8000)\\ 
				&= \Big(j(z) - j(2z)\Big )\Big(j(z) - j\Big (\frac{z}{2}\Big) \Big)\Big(j(z) - j\Big(\frac{z+1}{2}\Big)\Big).
			\end{split}
		\end{align}
		
		The set
		\[\Big \{\Big(1728, -3375, 8000, j(2z), j\Big(\frac{z}{2}\Big), j\Big(\frac{z+1}{2}\Big), j(z)\Big) : z \in \h\Big\}\]
		is thus a multiplicative special curve in $\C^7$.
		
		For an example of a $7$-tuple of singular moduli lying on this curve, take 
		\[z = \frac{-1 + \sqrt{163}i}{2}.\] 
		Then
		\[ j(z) = -262537412640768000 = -2^{18} \cdot 3^3 \cdot 5^3 \cdot 23^3 \cdot 29^3,\]
		which we denote by $k$, is a singular modulus of discriminant $-163$. In this case, $j(2z), j(z/2), j((z+1)/2)$ are the three singular moduli of discriminant $-652$. These are respectively the roots $r, s, \bar{s}$ of the irreducible polynomial
		\begin{align*}
			&X^3 - 68925893036109279891085639286946000 X^2\\ 
			&+ 102561728837719322645921325412908000000 X\\ 
			&-18095625621665522953693950872675200892692248000000000,
		\end{align*}
		where $r \in \R$ and $s, \bar{s}$ are complex conjugate with $s \in \h$. In this case, \eqref{eq:phi2} yields that
		\begin{align}\label{eq:egbig}
			-(k -1728)(k+3375)^2(k-8000)=(k-r)(k-s)(k-\bar{s}).
		\end{align}
		The prime factorisation of the two sides of \eqref{eq:egbig} is given by
		\begin{align*}
			&-2^{12} \cdot 3^{22} \cdot 5^{9} \cdot 7^6 \cdot 11^2 \cdot 13^3 \cdot 17^2 \cdot 19^2 \cdot 31^2 \cdot 37 \cdot 101 \cdot 103^2 \cdot 127^2\\ 
			&\cdot 157 \cdot 163 \cdot 229^2 \cdot 277 \cdot 283^2 \cdot 317.
		\end{align*} 	
	\end{example}

	\subsection{Multiplicative properties of differences of singular moduli}
	
	The study of the multiplicative properties of differences of singular moduli goes back at least as far as Berwick \cite{Berwick28}, who in 1927 determined the factorisations of $x$ and $x -1728$ for all singular moduli $x$ such that $[\Q(x) : \Q] \leq 3$. 
	
	The differences of singular moduli are highly divisible numbers, in the sense that they tend to have relatively small prime factors. For example,
	\[ j\Big(\frac{-1 + \sqrt{163} i}{2}\Big) - j\Big(\frac{-1 + \sqrt{67} i}{2}\Big) = -2^{15} \cdot 3^7 \cdot 5^3 \cdot 7^2 \cdot 13 \cdot 139 \cdot 331.\] 
	Example~\ref{eg:mod} gives another illustration of this tendency. This observation led Gross and Zagier \cite{GrossZagier85} to prove a formula for the prime ideal factorisations of differences of singular moduli, subject to some restrictions on the discriminants of the singular moduli considered. A version of their result for arbitrary discriminants has since been proved by Lauter and Viray \cite{LauterViray15}.
	
	Recent work on multiplicative relations among singular moduli, for example the proof of Theorem~\ref{thm:fixedy} by Pila and Tsimerman \cite{PilaTsimerman17} and the third author \cite{Fowler21}, has been motivated by connections to the Zilber--Pink conjecture on atypical intersections. 
	
	Effective results on multiplicative relations among singular moduli in low dimensions have also been studied extensively \cite{BiluGunTron22, BiluLucaMadariaga16, Fowler20, Fowler23, Riffaut19} as special cases of the Andr\'e--Oort conjecture for $\C^n$, which was proved ineffectively by Pila \cite{Pila11}. In particular, for $n \leq 3$, explicit bounds on multiplicatively dependent $n$-tuples of pairwise distinct singular moduli are known \cite{BiluGunTron22, Riffaut19}.
	
	For differences of singular moduli, the most general effective result we are aware of is Li's result \cite{Li21} that the difference of two singular moduli is never an algebraic unit.  Li's result generalised Bilu, Habegger, and K\"uhne's theorem \cite{BiluHabeggerKuhne18} that no singular modulus is a unit. Work on this topic was prompted by a question of Masser, answered affirmatively by Habegger \cite{Habegger15}, as to whether only finitely many singular moduli are algebraic units.

	\subsection{Structure of this paper}\label{subsec:struc}
	
	In Section~\ref{sec:prelim}, we give some of the basic results we will need for this paper. Section~\ref{sec:fixedy} completes the proof of Theorem~\ref{thm:fixedy}. The proof of Theorem~\ref{thm:MSC} is in Section~\ref{sec:MSC}. Section~\ref{sec:Ax} contains the functional transcendence results which are required for the proof of Theorem~\ref{thm:main}, which is then carried out in Section~\ref{sec:pf}. Finally, the connection to the Zilber--Pink conjecture is considered in Section~\ref{sec:ZP}.
	
	\vspace{5mm}
	\noindent
	{\sc Acknowledgements.} The authors would like to thank Gabriel Dill for helpful comments.
	
	\section{Preliminaries}\label{sec:prelim}
	
	\subsection{The fundamental domain}\label{subsec:fund}
	
	The group $\SL$ is generated by the matrices corresponding to the transformations $T \colon z \mapsto z+1$ and $S \colon z \to -1/z$. Let $\mathfrak{F}_j$ be the fundamental domain for the action of $\SL$ on $\h$ given by
	\[\Big \{ z \in \h : \re z \in \Big[-\frac{1}{2}, \frac{1}{2}\Big), \lvert z \rvert \geq 1, \mbox{ and if } \lvert z \rvert =1, \mbox{ then } \re z \in \Big[-\frac{1}{2}, 0\Big]\Big\}.\]
	This is a hyperbolic triangle with corners at $\rho, -\bar{\rho}, i \infty$. The $j$-function restricts to a bijection $j \colon \mathfrak{F}_j \to \C$.
	
	The $j$-function has a series expansion
	\[ j(z) = e^{-2 \pi i z} + 744 + \sum_{n=1}^\infty c(n) e^{2 n \pi i z},\]
	where the coefficients $c(n) \in \Z$. It follows immediately that the $j$-function is real valued on $\mathfrak{F}_j$ only along the boundary of $\mathfrak{F}_j$ and on the imaginary axis. Further, the image under $j$ of the set $\{z \in \mathfrak{F}_j : \lvert z \rvert =1\}$ is the real interval $[0, 1728]$.
	
	\begin{figure}
		
		\centering
		
		\begin{tikzpicture}[scale=2]
			\draw[-latex](\myxlow-0.1,0) -- (\myxhigh+0.2,0) node[below,font=\tiny] {$\re z$};
			\pgfmathsetmacro{\succofmyxlow}{\myxlow+0.5}
			
			\foreach \y  in {1}
			{   \draw (0,\y) -- (-0.05,\y) node[below ,font=\tiny]  {$i$ 
				};
			}
			
			{\draw (-0.5,0.866) -- (-0.5,0.866) node[left,font=\tiny] {$\rho$};
			}
			
			{\draw (0.5,0.866) -- (0.5,0.866) node[right,font=\tiny] 
				{$-\bar{\rho}$};
			}

			\draw[-latex](0,-0.1) -- (0,3.4) node[right,font=\tiny] {$\im z$};
			\draw[very thin, blue](-1.5,0) -- (-1.5,3);
			\draw[very thin, blue](-0.5,0) -- (-0.5,0.866);
			\draw[ blue, thick](-0.5,0.866) -- (-0.5,3);
			\draw[very thin, blue](0.5,0) -- (0.5,0.866);
			\draw[dashed, blue, thick](0.5,0.866) -- (0.5,3);
			\draw[very thin, blue](1.5,0) -- (1.5,3);
			
			\draw[very thin, blue] (0,0) arc(0:180:1);
			\draw[very thin, blue] (2,0) arc(0:180:1);
			\draw[very thin, blue] (-1,0) arc(0:90:1);
			\draw[very thin, blue] (1,0) arc(180:90:1);
			\draw[very thin, blue] (1,0) arc(0:60:1);
			\draw[very thin, blue] (-0.5,0.866) arc(120:180:1);
			
			\begin{scope}   
				\clip (\myxlow,0) rectangle (\myxhigh,1.1);
				\foreach \i in {2,...,\myiterations}
				{   \pgfmathsetmacro{\mysecondelement}{\myxlow+1/pow(2,floor(\i/3))}
					\pgfmathsetmacro{\myradius}{pow(1/3,\i-1}
					\foreach \x in {-2,\mysecondelement,...,2}
					{   \draw[very thin, blue] (\x,0) arc(0:180:\myradius);
						\draw[very thin, blue] (\x,0) arc(180:0:\myradius);
					}   
				}
			\end{scope}

			\draw [dashed, blue,thick,domain=60:90] plot ({cos(\x)}, {sin(\x)});
			\draw [blue,thick,domain=90:120] plot ({cos(\x)}, {sin(\x)});

			\begin{scope}
				\begin{pgfonlayer}{background}
					\clip (-0.5,0) rectangle (0.5,2.8);
					\clip   (1,2.8) -| (-1,0) arc (180:0:1) -- cycle;
					\fill[gray!20] (-1,-1) rectangle (1,2.8);
				\end{pgfonlayer}
			\end{scope}
		\end{tikzpicture}
		\caption{The fundamental domain $\mathfrak{F}_j$}
	\end{figure}
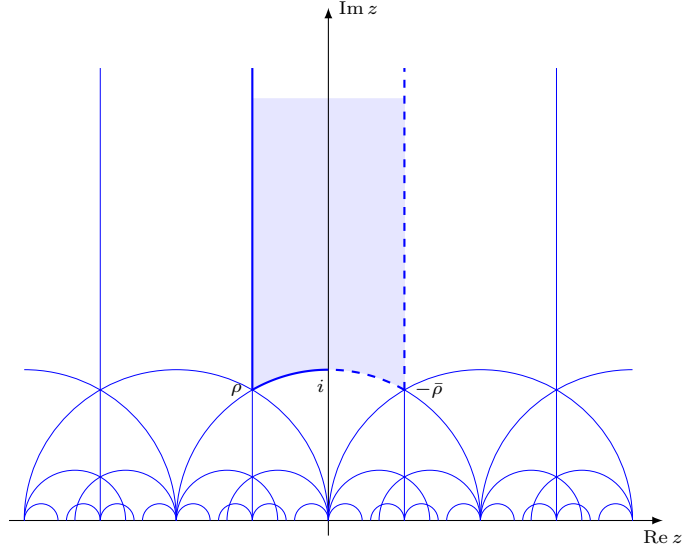

	\begin{prop}\label{prop:fundom}
		Let $z_0 \in \mathfrak{F}_j$. If $\re z_0 \neq -1/2$, then the $\SL$-orbit of $z_0$ is equal to
		\begin{align*}
			&\{z_0 + k : k \in \Z\} \cup \Big \{\frac{-1}{z_0} + k : k \in \Z\Big\}\\ &\cup \Big\{w \in \h : w \in \mathrm{Orbit}(z_0) \mbox{ and } \im w < \im \frac{-1}{z_0}\Big\}.
		\end{align*}
		If $\re z_0 = -1/2$, then the $\SL$-orbit of $z_0$ is equal to
		\begin{align*} 
			&\{z_0 + k : k \in \Z\} \cup \Big\{\frac{-1}{z_0} + k : k \in \Z\Big\} \cup \Big\{\frac{-1}{z_0+1} + k : k \in \Z\Big\}\\
			&\cup \Big\{w \in \h : w \in \mathrm{Orbit}(z_0) \mbox{ and } \im w < \im \frac{-1}{z_0}\Big\}.
		\end{align*}
	\end{prop}
	
	\begin{proof}
		First, we claim that the following algorithm applied to a point $z \in \h$ will output the unique point in $\mathfrak{F}_j \cap \mathrm{Orbit}(z)$.
		
		\begin{enumerate}
			\item If $z \in \mathfrak{F}_j$, then output $z$. Otherwise proceed to step (2).
			\item Replace $z$ with $z + k$, where $k \in \Z$ is such that $\re (z+k) \in [-1/2, 1/2)$.
			\item If $z \in \mathfrak{F}_j$, then output $z$. Otherwise proceed to step (4).
			\item Replace $z$ with $-1/z$. Return to step (1).
		\end{enumerate}
		
		Clearly, if this algorithm terminates, then it outputs the unique point in $\mathfrak{F}_j$ in the same $\SL$-orbit as the initial input. We claim that this algorithm always terminates. To prove this, note that $S \colon z \mapsto -1/z$ sends $z$ to
		\[ -\frac{\re z}{\lvert z \rvert^2} + \frac{\im z}{\lvert z \rvert^2} i.\]
		In particular, if $\lvert z \rvert < 1$, then the imaginary part of $-1/z$ is strictly larger than $\im z$. Now every application of step (4) is performed on some $z$ with $\lvert z \rvert \leq 1$. And if $\lvert z \rvert = 1$, then applying (4) immediately yields a point in $\mathfrak{F}_j$. 
		
		Hence, it suffices to prove that, given $z \in \h$ with $\re z \in [-1/2, 1/2)$, there are only finitely many $\gamma \in \SL$ with $\re \gamma z \in [-1/2, 1/2)$ and $\im \gamma z > \im z$. Write 
		\[ \gamma = \begin{pmatrix}
			a & b\\
			c & d
		\end{pmatrix}.
		\]
		Then
		\[\im \gamma z = \im \frac{az+b}{cz + d} = \frac{1}{\lvert cz + d \rvert^2} \im z,\]
		and so $\im \gamma z > \im z$ implies that
		\[ \lvert cz + d \rvert^2 < 1.\]
		Hence,
		\[ (c \re z + d)^2 + (c \im z)^2 < 1.\]
		Hence, there are only finitely many possibilities for $c$, and for each such $c$, only finitely many possibilities for $d$. So we may assume that $c, d$ are fixed.
		
		We now show that, for the pair $(c, d)$, there exists a unique pair $(a, b)$ such that
		\[ \begin{pmatrix}
			a & b\\
			c & d
		\end{pmatrix} \in \SL\]
		and
		\[\re \Big(\begin{pmatrix}
			a & b\\
			c & d
		\end{pmatrix} z\Big) \in \Big[-\frac{1}{2}, \frac{1}{2}\Big).\]
		Suppose $a, b, a', b'$ are such that 
		\[ \begin{pmatrix}
			a & b\\
			c & d
		\end{pmatrix}, \begin{pmatrix}
			a' & b'\\
			c & d
		\end{pmatrix} \in \SL.
		\]
		So $ad-bc =1$ and $a'd-b'c=1$. Hence, by B\'ezout's Lemma, there exists $k \in \Z$ such that 
		\[ (a', b') = (a+kc, b+kd).\]
		Thus,
		\[\begin{pmatrix}
			a' & b'\\
			c & d
		\end{pmatrix}z = \frac{(a+kc)z+(b+kd)}{cz + d} = \frac{az+b}{cz + d} + k.\]
		In particular, there is a unique $k \in \Z$ (and hence a unique pair $(a', b')$) such that
		\[ \re \Big(\begin{pmatrix}
			a' & b'\\
			c & d
		\end{pmatrix}z\Big) \in \Big[-\frac{1}{2}, \frac{1}{2}\Big).\]
		
		Thus, the algorithm always terminates. We may now complete the proof of the proposition. Let $z_0 \in \mathfrak{F}_j$. The proposition amounts to classifying all the $w \in \mathrm{Orbit}(z_0)$ such that
		\[ \im w \geq \im \frac{-1}{z_0}.\]
		
		Suppose first that $\re z_0 \neq -1/2$. Let $w_0 \in \mathrm{Orbit}(z_0)$ be such that 
		\[ w_0 \notin \{z_0 + k : k \in \Z\} \cup \Big\{\frac{-1}{z_0} + k : k \in \Z\Big\}.\]
		We claim that $\im w_0 < \im (-1/z_0)$. Applying the above algorithm to $w_0$, we must obtain $z_0$ after finitely many steps. The last transformation applied is either $z \mapsto -1/z$ or $z \mapsto z+k$ for some $k \in \Z \setminus \{0\}$. Recall that $z \mapsto -1/z$ is its own inverse.
		
		If the last transformation applied is $z \mapsto -1/z$, then the algorithm applied to $w_0$ must pass through $-1/z_0$. By assumption on $w_0$, this must happen after an application of $z \mapsto -1/z$, which must have strictly increased the imaginary part.
		
		If the last transformation applied is $z \mapsto z+k$ with $k \neq 0$, then the transformation prior to that must have been $z \mapsto -1/z$. We thus must have that
		\[ \im w_0 \leq \im \frac{-1}{z_0 - k} = \frac{1}{\lvert z_0 - k \rvert^2} \im z_0.\]
		Then 
		\[ \lvert z_0 - k \rvert > \lvert z_0 \rvert \geq 1,\]
		since $\re z_0 \in (-1/2, 1/2)$ and $k \neq 0$. Hence, 
		\[ \im w_0 < \frac{1}{\lvert z_0 \rvert^2} \im z_0 = \im \frac{-1}{z_0}\leq \im z_0.\]
		
		Now let $z_0 \in \mathfrak{F}_j$ be such that $\re z_0 = -1/2$. Let $w_0 \in \mathrm{Orbit}(z_0)$ be such that 
		\[ w_0 \notin \{z_0 + k : k \in \Z\} \cup \Big\{\frac{-1}{z_0} + k : k \in \Z\big\}\cup \Big\{\frac{-1}{z_0+1} + k : k \in \Z\Big\}.\]
		We claim that $\im w_0 < \im (-1/z_0)$. To show this, we repeat the above argument. The only place where $\re z \neq -1/2$ was used was to obtain the inequality 
		\[ \lvert z_0 - k \rvert > \lvert z_0 \lvert.\]
		If $k \neq -1$, then this inequality still holds and the above argument works. So assume $k = -1$. Then, by the assumption that
		\[w_0 \notin \Big\{\frac{-1}{z_0+1} + k : k \in \Z \Big\},\]
		there must have been an application of $z \mapsto -1/z$ prior to passing through $-1/(z_0 + 1)$ and this application must have strictly increased the imaginary part. Hence,
		\[ \im w_0 < \frac{1}{\lvert z_0 +1 \rvert^2} \im z_0 = \im \frac{-1}{z_0}\leq \im z_0.\]
		This completes the proof.
	\end{proof}

	\subsection{Singular moduli}\label{subsec:singmods}
	
	Let $\tau \in \h$ be such that $[\Q(\tau) : \Q]=2$. So
	\[ a \tau^2 + b \tau + c =0\]
	for some $(a,b,c) \in \Z^3 \setminus \{(0,0,0)\}$ with $\gcd(a,b,c)=1$. The discriminant of the singular modulus $j(\tau)$ is defined to be
	\[ b^2 - 4ac.\]
	This depends only on the value of $j(\tau)$ and not on the choice of $\tau$. For a singular modulus $x$, write $\Delta(x)$ for the discriminant of $x$. The singular moduli of a given discriminant $\Delta$ form a complete set of $\Q$-conjugates \cite[Proposition~13.2]{Cox89}.
	
	\begin{lem}\label{lem:fixedpt}
		Suppose that $z \in \h$ is such that $g z = z$ for some $g \in \MZ$ such that $\det g > 0$ and $\lambda g \neq \mathrm{Id}_2$ for every $\lambda \in \Q^\times$. Then $j(z)$ is a singular modulus and $\lvert \Delta(j(z)) \rvert \leq 4 \det g$.
	\end{lem}
	
	\begin{proof}
		Let $N = \det g$. Write
		\[g = \begin{pmatrix}
			a & b\\
			c & d
		\end{pmatrix}.\]
		So $ad-bc = N$. Since $z$ is a fixed point, we have that
		\[z = \frac{az+b}{cz+d}.\]
		Thus, 
		\[cz^2 + (d-a)z - b =0.\]
		If $b=c=d-a=0$, then $N = a^2$ and $g = a \mathrm{Id}_2$, which is excluded. So some coefficient of this quadratic equation is non-zero. Hence, $j(z)$ is a singular modulus. Let $h = \gcd(c, -b, (d-a))$. Then 
		\[\Delta(j(z)) = \Big(\frac{d-a}{h}\Big)^2 +  \frac{4bc}{h^2}.\]
		Since $bc = ad - N$, we have that
		\[\Delta(j(z)) = \frac{1}{h^2}((a+d)^2 - 4N).\]
		In particular, $\lvert \Delta(j(z)) \rvert \leq 4N$.
	\end{proof}
	
	\begin{prop}\label{prop:eqcoordsy}
		Suppose that $z \in \h$ is such that
		\[j(z) = j(g z)\]
		for some $g \in \MZ$ such that $\det g >0$ and $\lambda g \notin \SL$ for every $\lambda \in \Q^\times$. Then $j(z)$ is a singular modulus and $\lvert \Delta(j(z)) \rvert \leq 4 \det g$.
	\end{prop}
	
	\begin{proof}
		Since 
		\[j(z) = j(g z),\]
		there exists $\gamma \in \SL$ such that
		\[\gamma z = gz.\]
		In particular, $z$ is a fixed point for the action of the integer matrix $\gamma^{-1} g$ on $\h$. Apply Lemma~\ref{lem:fixedpt} to $\gamma^{-1} g$.
	\end{proof}
	
	\begin{cor}\label{cor:rootsofF}
		Let $x \in \C$ be a root of the polynomial $F_N$ for some $N \in \Z_{>1}$. Then $x$ is a singular modulus and $\lvert \Delta(x) \rvert \leq 4N$. Further, if $N \in \Z_{>1}$, then every singular modulus of discriminant $-4N$ is a root of $F_N$.
	\end{cor}
	
	\begin{proof}
		Suppose $N \in \Z_{>1}$ and $x \in \C$ are such that $F_N(x)=0$. Recall that
		\[F_N(j(z)) = \prod_{g \in C(N)} (j(z) - j(gz)).\]
		The $j$-function is surjective, so there exists $z_0 \in \h$ and $g \in \C(N)$ such that $j(z_0)=j(gz_0) = x$. Then, by Proposition~\ref{prop:eqcoordsy}, we have that $x$ is a singular modulus and $\lvert \Delta(x) \rvert \leq 4N$.
		
		For the second part, note that for $N \in \Z_{>1}$,
		\[ \begin{pmatrix}
			1 & 0\\
			0 & N
		\end{pmatrix} \in C(N).\]
		Since the $j$-function is invariant under $z \mapsto -1/z$, we have that
		\[j(\sqrt{N} i) = j\Big(\frac{1}{N} \sqrt{N} i\Big).\]
		So $F_N(j(\sqrt{N}i))=0$. Clearly, $j(\sqrt{N}i)$ is a singular modulus of discriminant $-4N$. Recall that the singular moduli of discriminant $-4N$ are all conjugate over $\Q$. Thus, every singular modulus of discriminant $-4N$ is a root of $F_N$, since $F_N$ has coefficients in $\Z$.
	\end{proof}
	
	\begin{cor}\label{cor:vanish}
		Let $N_1, \ldots, N_k \in \Z_{>1}$ be pairwise distinct. Let $b_1, \ldots, b_k \in \Z \setminus \{0\}$. Then 
		\[\prod_{i=1}^k F_{N_i}(X)^{b_i}\]
		is a non-constant rational function of $X$.
	\end{cor}
	
	\begin{proof}
		Each $F_{N_i}(X)$ is a non-constant polynomial in $X$. Hence, the rational function
		\[\prod_{i=1}^k F_{N_i}(X)^{b_i}\]
		is not constantly zero. Without loss of generality, we may assume that $N_k > N_1, \ldots, N_{k-1}$ and $b_k > 0$. By Corollary~\ref{cor:rootsofF}, the polynomial $F_{N_i}(X)$ vanishes at a singular modulus of discriminant $-4N_k$ if and only if $i=k$. Thus the rational function
		\[\prod_{i=1}^k F_{N_i}(X)^{b_i}\]
		vanishes at every singular modulus of discriminant $-4N_k$; in particular, the function is non-constant.
	\end{proof}
	
	\begin{prop}\label{prop:eqcoords}
		Suppose that $g_1, g_2 \in \MZ$ are such that $\det g_1, \det g_2 > 0$ and $g_1 \neq \lambda \gamma g_2$ for every $\lambda \in \Q^\times$ and $\gamma \in \SL$. If $z \in \h$ is such that $j(g_1 z) = j(g_2 z)$, then $j(z)$ is a singular modulus and $\lvert \Delta(j(z)) \rvert \leq 4 \det(g_1)\det(g_2)$. 
	\end{prop}
	
	\begin{proof}
		Since $j(g_1 z) = j(g_2 z)$, there exists $\gamma \in \SL$ such that $\gamma g_1 z = g_2 z$. Hence, $z$ is a fixed point for the action of $g = g_2^{-1} \gamma g_1 \in \GL$ on $\h$. Multiplying the entries of $g$ by $\det(g_2)$, we may assume that $g \in \MZ$ and $\det g \leq \det(g_1) \det(g_2)$. Since $g_1 \neq \lambda \gamma' g_2$ for every $\lambda \in \Q^\times$ and $\gamma' \in \SL$, $g$ is not a rational scalar multiple of $\mathrm{Id}_2$. The desired result thus follows from Lemma~\ref{lem:fixedpt}.
	\end{proof}
	
	\begin{prop}\label{prop:galbd}
		For every $\epsilon > 0$, there exist an ineffective constant $c_1(\epsilon)>0$ and an effective constant $c_2(\epsilon)>0$, such that if $x$ is a singular modulus of discriminant $\Delta$, then
		\[ [\Q(x) : \Q] \geq c_1(\epsilon) \lvert \Delta \rvert^{1/2 - \epsilon}\]
		and
		\[ [\Q(x) : \Q] \leq c_2(\epsilon) \lvert \Delta \rvert^{1/2 + \epsilon}.\]
	\end{prop}
	
	\begin{proof}
		The ineffective lower bound is due to Siegel \cite{Siegel35}. The upper bound is \cite[Proposition~2.2]{Paulin16}.
	\end{proof}
	
	Since the restriction $j \colon \mathfrak{F}_j \to \C$ of the $j$-function to the fundamental domain is bijective, the map
	\[ (a, b, c) \mapsto j\Big(\frac{-b + \lvert \Delta \rvert^{1/2} i}{2a}\Big)\]
	is a bijection between the set 
	\begin{align*}
		T_\Delta = \Big\{(a, b, c) \in \Z^3: \, &\Delta = b^2 -4ac, \, \gcd(a,b,c) = 1,\\ &\mbox{ and either } -a<  b \leq a <c \mbox{ or } 0 \leq b \leq a =c\Big\}
	\end{align*}
	and the singular moduli of discriminant $\Delta$. Observe that, for each discriminant $\Delta$, there is a unique triple $(a,b,c) \in T_\Delta$ with $a=1$. This triple is given by $(1, k, (k^2 - \Delta)/4)$, where $k=0$ if $\Delta$ is even and $k=1$ if $\Delta$ is odd. The corresponding singular modulus has preimage 
	\[ \frac{-k + \lvert \Delta \rvert^{1/2} i}{2} \in \mathfrak{F}_j,\]
	which has imaginary part strictly greater than the preimage of any other singular modulus of discriminant $\Delta$ and of any singular modulus of discriminant $\Delta'$ with $\lvert \Delta' \rvert < \lvert \Delta \rvert$.
	
	For $\alpha \in \alg$, write $H(\alpha)$ for the absolute multiplicative height of $\alpha$ and $h(\alpha)$ for the absolute logarithmic height (see e.g.~\cite[\S1.5]{BombieriGubler06}).
	
	\begin{prop}\label{prop:htpresingmod}
		Let $x$ be a singular modulus of discriminant $\Delta$. Let $\tau \in \mathfrak{F}_j$ be such that $j(\tau) = x$. Then
		\[H(\re \tau) \leq \frac{2 \lvert \Delta \rvert^{1/2}}{\sqrt{3}} \mbox{ and } H(\im \tau) \leq \frac{4 \lvert \Delta \rvert}{3}.\]
	\end{prop}
	
	\begin{proof}
		By the above characterisation of the singular moduli of discriminant $\Delta$, we have that
		\[ \tau = \frac{-b + \lvert \Delta \rvert^{1/2}i}{2a}\]
		for some $(a, b, c) \in \Z^3$ with $0 \leq \lvert b \rvert \leq a \leq c$ and $b^2 - 4 ac = \Delta$. Hence, by e.g. \cite[Propositions~1.6.5 and 1.6.6]{BombieriGubler06},
		\[ H(\re \tau) = \max\{\lvert b \rvert, 2 a\} = 2a \mbox{ and } H(\im \tau) = \max \{ \lvert \Delta \rvert, 4 a^2\}.\]
		The desired inequality follows, since $3a^2 \leq 4ac - b^2 = \lvert \Delta \rvert$. Finally, observe that if $\Delta = - 3$, then $\tau = (-1 + \sqrt{3}i)/2$ and both bounds are achieved.
	\end{proof}

	\begin{prop}[{\cite[Lemma~4.3]{HabeggerPila12}}]\label{prop:htsingmod}
		For every $\epsilon > 0$, there exists an ineffective constant $c(\epsilon)>0$ such that if $x$ is a singular modulus of discriminant $\Delta$, then
		\[ h(x) \leq c(\epsilon) \lvert \Delta \rvert^{\epsilon}.\]
	\end{prop}
	
	\subsection{Properties of $j$-maps}\label{sec:jmap}
	
	Let $f$ be a non-constant $j$-map. Then \cite[Proposition~7.1]{BiluLucaMasser17} there exist $r, s \in \Q$ such that $r>0$ and $0 \leq s < 1$ such that $f(z) = j(rz +s)$ for all $z \in \h$. Two non-constant $j$-maps are equal if and only if the corresponding pairs $(r, s)$ are equal.
	
	Recall that, for $N \in \Z_{>0}$, we define
	\[ C(N) = \Big \{ \begin{pmatrix}
		a & b\\
		0 & d
	\end{pmatrix} \in \MZ : ad=N, a>0, 0 \leq b < d, \gcd(a, b, d) = 1 \Big \}.\] 
	
	\begin{prop}\label{prop:unique}
		Let $g \in \GL$. Then there exist a unique $N \in \Z_{>0}$ and a unique $g' \in C(N)$ such that $j(g z) = j(g' z)$ for all $z \in \h$.
	\end{prop}
	
	\begin{proof}
		Let $g \in \GL$. Then there exist $r, s \in \Q$ with $r>0$ and $0 \leq s < 1$ such that $f(z) = j(rz +s)$ for all $z \in \h$. Further, the pair $(r, s)$ is unique. Let $\lambda \in \Q_{>0}$ be such that
		\[ \lambda \begin{pmatrix}
			r & s\\
			0 & 1
		\end{pmatrix} = \begin{pmatrix}
			a & b\\
			0 & d
		\end{pmatrix}
		\]
		for some $a, b, d \in \Z$ with $\gcd(a, b, d) =1$. Since $0 \leq s < 1$, we have that $0 \leq b < d$. Let $N = ad$. Then 
		\[ \begin{pmatrix}
			a & b\\
			0 & d
		\end{pmatrix} \in C(N),\]
		and
		\[j(gz) = j\Big(\begin{pmatrix}
			a & b\\
			0 & d
		\end{pmatrix} z\Big)\]
		for all $z \in \h$.
		
		Suppose 
		\[
		\begin{pmatrix}
			a' & b'\\
			0 & d'
		\end{pmatrix} \in C(M)
		\]
		were also such that 
		\[j\Big(\frac{a}{d} z+ \frac{b}{d}\Big) = j\Big(\frac{a'}{d'} z+ \frac{b'}{d'}\Big)\]
		for all $z \in \h$. Then, by the uniqueness of the representation of a $j$-map in terms of $r$ and $s$, we would have that
		\[ \frac{a}{d} = \frac{a'}{d'} \mbox{ and } \frac{b}{d} = \frac{b'}{d'}. \]
		Hence,
		\[\frac{a}{a'}  = \frac{d}{d'},\]
		and either $b= b' = 0$ or
		\[\frac{b}{b'} = \frac{d}{d'}.\]
		So one matrix is just the rescaling of the other. Since the entries of each of the two matrices are coprime integers and $a, a'>0$, the two matrices are in fact identical and $M=N$.
	\end{proof}
	
	\begin{prop}\label{prop:jmaps}
		Let $N \in \Z_{>0}$.
		\begin{enumerate}
			\item For every $\gamma \in \SL$ and $g \in C(N)$, there exists $h \in C(N)$ such that $j(g \gamma z) = j(h z)$.
			\item For every $\gamma \in \SL$ and $g, h \in C(N)$, if $g \neq h$, then $j(g \gamma z) \neq j(h \gamma z)$.
			\item For every $g, h \in C(N)$, there exists $\gamma \in \SL$ such that $j(g \gamma z) = j(h z)$.
		\end{enumerate}
	\end{prop}

	\begin{proof}
		Let $\gamma \in \SL$ and $g \in C(N)$. The entries of $g$ are coprime integers and $\det g = N$. Since $\gamma \in \SL$, the entries of $g \gamma$ are coprime integers and $\det g \gamma = N$. Write $g \gamma$ as 
		\[ \begin{pmatrix}
			a & b\\
			c & d
		\end{pmatrix}\]
		for $a,b,c,d \in \Z$ with $\gcd(a,b,c,d)=1$ and $N = ad-bc$. Let $\mu = \gcd(a, c)$ and $m, n \in \Z$ be such that $\mu = ma + nc$. Then
		\[\underbrace{\begin{pmatrix}
				m & n\\
				-c/\mu & a/\mu
		\end{pmatrix}}_{\in \SL}
		\begin{pmatrix}
			a & b\\
			c & d
		\end{pmatrix} =
		\begin{pmatrix}
			\mu & mb + nd\\
			0 & -bc/\mu+ad/\mu
		\end{pmatrix},
		\]
		which is upper triangular with coprime integer entries and has determinant $N$. 
		
		Let $p=mb+nd$ and $q = N/\mu$. Then
		\[j(g \gamma z) = j\Big(\begin{pmatrix}
			a & b\\
			c & d
		\end{pmatrix}z\Big) = j\Big(\begin{pmatrix}
			\mu & p\\
			0 & q
		\end{pmatrix} z\Big) = j\Big(\begin{pmatrix}
			\mu & p + kq\\
			0 & q
		\end{pmatrix}z\Big)\]
		for all $z \in \h$, where $k \in \Z$ is the unique integer such that $0 \leq p + kq < q$. This proves (1), since
		\[ \begin{pmatrix}
			\mu & p + kq\\
			0 & q
		\end{pmatrix} \in C(N).\]
		
		Given Proposition~\ref{prop:unique}, (2) follows immediately by making the change of variables $w = \gamma^{-1} z$.

		Now we prove (3). Let 
		\[ \Gamma_0(N) = \Big\{ \begin{pmatrix}
			a & b\\
			c& d
		\end{pmatrix} \in \SL : c \equiv 0 \bmod N\Big\}.\] 
		Let 
		\[ \sigma_{N} = \begin{pmatrix}
			N & 0\\
			0 & 1
		\end{pmatrix}.\]
		Observe that $\sigma_N \in C(N)$. By \cite[Lemma~11.11]{Cox89}, the map 
		\[ g \mapsto \sigma_{N}^{-1} \SL g \cap \SL \]
		gives a bijection between the elements $g \in C(N)$ and the right cosets of $\Gamma_0(N)$ in $\SL$. In particular, for every $g \in C(N)$, the set $\sigma_{N}^{-1} \SL g \cap \SL$ is non-empty.
		
		Let $g, h \in C(N)$. Let $\gamma_1 \in \sigma_{N}^{-1} \SL h \cap \SL$ and $\gamma_2 \in \sigma_{N}^{-1} \SL g \cap \SL$. Let $\gamma_{1,1}, \gamma_{2, 1} \in \SL$ be such that 
		\[ \gamma_{1} = \sigma_{N}^{-1} \gamma_{1, 1} h\]  
		and
		\[ \gamma_{2} = \sigma_{N}^{-1} \gamma_{2, 1} g.\] 
		Then
		\[j(g \gamma_2^{-1} \gamma_1 z) = j(g g^{-1} \gamma_{2, 1}^{-1} \sigma_{N} \sigma_{N}^{-1} \gamma_{1,1} h z)
		=j(\gamma_{2,1}^{-1} \gamma_{1,1} h z)
		=j(h z)\]
		for all $z \in \h$. So we may take $\gamma =  \gamma_2^{-1} \gamma_1$ in (3).
	\end{proof}
	
	\section{Completing the proof of Theorem~\ref{thm:fixedy}}\label{sec:fixedy}
	
	\begin{definition}\label{def:multdep}
		Let $n \in \Z_{>0}$ and $x_1, \ldots, x_n \in \C^\times$ be pairwise distinct. The set $\{x_1, \ldots, x_n\}$ is multiplicatively dependent if there exist $a_1, \ldots, a_n \in \Z$, not all zero, such that
		\[ \prod_{i=1}^n x_i^{a_i} = 1.\]
		The set $\{x_1, \ldots, x_n\}$ is minimally multiplicatively dependent if $\{x_1, \ldots, x_n\}$ is multiplicatively dependent and no non-empty proper subset of $\{x_1, \ldots, x_n\}$ is multiplicatively dependent.
	\end{definition}

	\begin{thm}\label{thm:fixedygd}
		Let $y \in \C$ be such that $y \notin (0, 1728)$. Let $n \in \Z_{>0}$. Then there exist only finitely many $n$-tuples $(x_1, \ldots, x_n)$ of singular moduli $x_1, \ldots, x_n$ such that $x_1, \ldots, x_n, y$ are pairwise distinct and $\{x_1-y, \ldots, x_n-y\}$ is minimally multiplicatively dependent. 
	\end{thm}
	
	We do not need to assume that $y$ is a singular modulus in Theorem~\ref{thm:fixedygd}, because the same proof works for all $y$ outside the real interval $(0, 1728)$.
	
	\begin{proof}
		Let $f(z) = j(z) - y$. Then the only zero of $f$ in $\mathfrak{F}_j$ is at the unique $\tau \in \mathfrak{F}_j$ such that $j(\tau) = y$. Since $y \notin (0, 1728)$, this point $\tau$ does not lie on the arc of the circle $\lvert z \rvert = 1$ strictly between $i$ and $\rho$. So
		\[ \im \frac{-1}{\tau} < \im \tau,\]
		and, by Proposition~\ref{prop:fundom}, $f(\tau + s) \neq 0$ for all $s \in (0, 1)$. Thus, $f$ satisfies the ``divisor condition'' of \cite[Definition~1.3]{Fowler21}, and hence \cite[Theorem~1.6]{Fowler21} implies the desired result.
	\end{proof}
	
	\begin{thm}\label{thm:fixedybad}
		Let $y$ be a singular modulus. Let $n \in \Z_{>0}$.  There exist only finitely many $n$-tuples $(x_1, \ldots, x_n)$ such that $x_1, \ldots, x_n, y$ are pairwise distinct singular moduli and $\{x_1-y, \ldots, x_n-y\}$ is minimally multiplicatively dependent. 
	\end{thm}
	
	\begin{proof}
		By Theorem~\ref{thm:fixedygd}, we may assume that $y \in (0, 1728)$. Let $\Delta = \Delta(y)$. Note that $\lvert \Delta \rvert > 4$, since $0, 1728$ are the only singular moduli with discriminant in the set $\{-3, -4\}$. In particular, $y$ has the $\Q$-conjugate 
		\[y' = j\Big(\frac{-k + \lvert \Delta \rvert^{1/2}i}{2}\Big),\]
		where $k=0$ if $\Delta$ is even and $k=1$ if $\Delta$ is odd. Since
		\[ \frac{\lvert \Delta \rvert^{1/2}}{2} > 1,\]
		we have that $y' \notin (0, 1728)$. Thus, Theorem~\ref{thm:fixedygd} holds for $y'$, and so Theorem~\ref{thm:fixedygd} for $y$ follows since $y, y'$ are conjugate over $\Q$.
	\end{proof} 
	
	Theorem~\ref{thm:fixedy} seems stronger than Theorem~\ref{thm:fixedybad}, since the former does not require the multiplicative dependence to be minimal, only that all the exponents are non-zero. In fact, we may deduce Theorem~\ref{thm:fixedy} from Theorem~\ref{thm:fixedybad} by the following formal argument.
	
	\begin{prop}\label{prop:mindep}
		Let $\mathcal{S} \subset \C^\times$. Let $n \in \Z_{>0}$. Suppose, for every $k \in \{1, \ldots, n\}$ there are only finitely many $k$-tuples $(s_1, \ldots, s_k) \in \mathcal{S}^k$ such that $s_1, \ldots, s_k$ are pairwise distinct and $\{s_1, \ldots, s_k\}$ is minimally multiplicatively dependent. Then there are only finitely many $n$-tuples $(s_1, \ldots, s_n) \in \mathcal{S}^n$ such that $s_1, \ldots, s_n$ are pairwise distinct and 
		\[ \prod_{i=1}^n s_i^{a_i} = 1\]
		for some $a_1, \ldots, a_n \in \Z \setminus \{0\}$.
	\end{prop}
	
	\begin{proof}
		Let $(s_1, \ldots, s_n) \in \mathcal{S}^n$ be such that $s_1, \ldots, s_n$ are pairwise distinct and 
		\[ \prod_{i=1}^n s_i^{a_i} = 1\]
		for some $a_1, \ldots, a_n \in \Z \setminus \{0\}$. The set $\{s_1, \ldots, s_n\}$ is thus multiplicatively dependent. For each $i \in \{1, \ldots, n\}$, there exists \cite[Lemma~5.9]{Fowler21} a minimally multiplicatively dependent subset $S_i \subset \mathcal{S}$ such that $s_i \in S_i$. In particular, $s_1, \ldots, s_n$ all belong to the set consisting of, for each $k \in \{1, \ldots, n\}$, all the coordinates of tuples $(s_1', \ldots, s_k') \in \mathcal{S}^k$ such that $s_1', \ldots, s_k'$ are pairwise distinct and the set $\{s_1', \ldots, s_k'\}$ is minimally multiplicatively dependent. By assumption, this set is finite and hence there are only finitely many possibilities for $(s_1, \ldots, s_n)$.
	\end{proof}
	
	\begin{proof}[Proof of Theorem~\ref{thm:fixedy}]
		Apply Proposition~\ref{prop:mindep} to Theorem~\ref{thm:fixedybad} with 
		\[ \mathcal{S} = \{ x - y : x \mbox{ is a singular modulus and } x \neq y\}.\qedhere\] \end{proof}

	\section{Multiplicative special curves}\label{sec:MSC}
	
	In this section, we prove Theorem~\ref{thm:MSC}. To do this, we first prove the following result.
	
	\begin{thm}\label{thm:MSCshape}
		Let $n \in \Z_{>0}$. Suppose that $T \subset \C^{n+1}$ is a multiplicative special curve. Then there exist $k \in \{1, \ldots, n\}$, $b_1, \ldots, b_k \in \Z \setminus \{0\}$, and pairwise distinct $N_1, \ldots, N_k \in \Z_{>1}$ such that, after reordering the first $n$ coordinates, we have that
		\[ T = \{(\alpha_1, \ldots, \alpha_m, j(g_1 z), \ldots, j(g_l z), j(z)) : z \in \h\},\]
		where
		\begin{enumerate}
			\item $\alpha_1, \ldots, \alpha_m$ are pairwise distinct and such that 
			\[ \{\alpha_1, \ldots, \alpha_m\} = \{\alpha \in \C : \alpha \mbox{ is either a zero or a pole of } \prod_{i=1}^k F_{N_i}(X)^{b_i} \};\]
			\item $g_1, \ldots, g_l \in \GL$ are pairwise distinct and such that
			\[ \{g_1, \ldots, g_l\} = \bigcup_{i=1}^k C(N_i). \]
		\end{enumerate}
	\end{thm}
	
	This follows immediately from the following result, which we will prove in Section~\ref{subsec:shape}. Throughout this paper, by a change of variables we mean replacing a variable $z$ by $gz$ for some $g \in \GL$.
	
	\begin{thm}\label{thm:multind}
		Let $n \in \Z_{>0}$. Let $f_1, \ldots, f_n, f$ be pairwise distinct $j$-maps, at least one of which is non-constant. Suppose that $a_1, \ldots, a_n \in \Z \setminus \{0\}$ and $c \in \C^\times$ are such that
		\begin{align}\label{eq:mult}
			\prod_{i=1}^n (f_i(z) - f(z))^{a_i} = c
		\end{align}
		for all $z \in \h$. Then, after a change of variables, $f(z) = j(z)$ and there exist $k \in \{1, \ldots, n\}$, $N_1, \ldots, N_k \in \Z_{>1}$ pairwise distinct, and $b_1, \ldots, b_k \in \Z \setminus \{0\}$ such that
		\[ \{ f_i : f_i \mbox{ is non-constant} \} = \{ j(g z) : g \in C(N_i), \, i =1,\ldots, k\},\]
		and, for all $z \in \h$,
		\[ \prod_{\substack{i \in \{1, \ldots, n\} \mbox{ s.t.}\\ f_i \mbox{ non-constant}}}  (f_i(z) - f(z))^{a_i} = \prod_{i=1}^k \Big (\prod_{g \in C(N_i)} (j(gz) - j(z)) \Big )^{b_i} \]
		and
		\[ \prod_{\substack{i \in \{1, \ldots, n\} \mbox{ s.t.}\\ f_i \mbox{ constant}}}  (f_i(z) - f(z))^{a_i} = c \prod_{i=1}^k F_{N_i}(j(z))^{-b_i}. \]
	\end{thm}

	\subsection{Functional independence modulo constants}
	
	Before proving Theorem~\ref{thm:multind}, we first prove some propositions using ideas from \cite[\S2]{Fowler21}. These will allow us to show that if 
	\begin{align}\label{eq:MSC}
		\{(f_1(z), \ldots, f_n(z), f(z)) : z \in \h\}
	\end{align}
	is a multiplicative special curve, then we must be in the situation that some $f_i$ is non-constant, $f$ is non-constant, and some $f_i$ is constant.
	
	\begin{definition}\label{def:funind}
		Functions $f_1, \ldots, f_n \colon \h \to \C$ are called multiplicatively independent modulo constants if, whenever $a_1, \ldots, a_n \in \Z$ are not all zero, the function $F \colon \h \to \C$ defined by
		\[ F(z) = \prod_{i=1}^n f_i(z)^{a_i}\]
		is non-constant.
	\end{definition}
	
	\begin{prop}\label{prop:multind1map}
		Let $n \in \Z_{>0}$. Let $f$ be a non-constant $j$-map. Let $\alpha_1, \ldots, \alpha_n \in \C$ be pairwise distinct. Then the functions $h_i(z) = f(z) - \alpha_i$ are multiplicatively independent modulo constants. 
	\end{prop}

	\begin{proof}
		By changing variables, we may assume that $f(z) = j(z)$. The result is then immediate since $j$ is a transcendental function.
	\end{proof}
	
	Thus, a multiplicative special curve as in \eqref{eq:MSC} must have at least one of $f_1, \ldots, f_n$ non-constant.
	
	\begin{prop}\label{prop:multindconstsing}
		Let $n \in \Z_{>0}$. Let $f_1, \ldots, f_n$ be pairwise distinct non-constant $j$-maps. Let $\alpha$ be a singular modulus. Then the functions $h_i(z) = f_i(z) - \alpha$ are multiplicatively independent modulo constants. 
	\end{prop}

	\begin{proof}
		Suppose, for contradiction, that $c \in \C^\times$ and $a_1, \ldots, a_n \in \Z \setminus \{0\}$ are such that
		\begin{align}\label{eq:multindconstsing}
			\prod_{i=1}^n (f_i(z) - \alpha)^{a_i} = c
		\end{align}
		for all $z \in \h$. If $[\Q(z) : \Q]=2$, then $f_i(z)$ is a singular modulus and so $f_i(z) \in \alg$. Hence, $c \in \alg$. Let $K= \Q(\alpha, c)$.
		
		We may write $f_i(z) = j(r_i z + s_i)$ for some $r_i, s_i \in \Q$ with $r_i >0$, $s_i \in [0, 1)$, and the pairs $(r_i, s_i)$ all distinct. Re-indexing and making a change of variables, we may assume that $f_1(z) = j(z)$ and $r_i \geq 1$ for $i \geq 2$. 
		
		For $k \in \Z_{>0}$, let $z_k = \sqrt{-k}$. Then $j(z_k)$ is a singular modulus of discriminant $-4k$ and every preimage under $j$ of every singular modulus of discriminant $-4k$ has imaginary part $\leq \sqrt{k}$ with equality precisely at the preimages of $j(z_k)$ itself which have the form $z_k + l$ for $l \in \Z$. This follows from the characterisation of the preimages of the singular moduli of a given discriminant which follows Proposition~\ref{prop:galbd} and the properties of the fundamental domain in Proposition~\ref{prop:fundom}.
		
		For $i > 1$, we thus have that $f_i(z_k)$ is a singular modulus with discriminant not equal to $-4k$. Also, the $f_i(z_k)$ are all pairwise distinct, by Proposition~\ref{prop:fundom}, and, if $k$ is large enough, not equal to $\alpha$. One thus has that
		\begin{align*}\label{eq:fi1}
			(j(z_k) - \alpha)^{a_1} \prod_{i=2}^n (x_i - \alpha)^{a_i} = c
		\end{align*}
		for some singular moduli $x_2, \ldots, x_n$ of discriminants not equal to $4k$.
		
		For all $k$ large enough, Proposition~\ref{prop:galbd} implies that the tuple
		\[(j(z_k), x_2, \ldots, x_n)\]
		has some Galois conjugate over $K$ of the form
		\[(\beta, x_2', \ldots, x_n'),\]
		where $\beta \neq j(z_k)$. Note that
		\begin{align*}\label{eq:fi2}
			(\beta - \alpha)^{a_1} \prod_{i=2}^n (x_i' - \alpha)^{a_i} = c.
		\end{align*}
		Thus,
		\[(j(z_k) - \alpha)^{a_1} \prod_{i=2}^n (x_i - \alpha)^{a_i} = (\beta - \alpha)^{a_1} \prod_{i=2}^n (x_i' - \alpha)^{a_i}. \]
		The only singular moduli of discriminant $-4k$ in this relation are $j(z_k)$ and $\beta$, and they are distinct. Hence, at least the terms $(j(z_k) - \alpha)$ and $(\beta - \alpha)$ in the above relation do not cancel. 
		
		Grouping the terms where $x_i = x_k'$, which we then cancel if $a_i = a_k$, we obtain, for some $m \in \{2, \ldots, 2n\}$, an $m$-tuple 
		\[(j(z_k), \beta, y_1, \ldots, y_{m-2})\] 
		of singular moduli such that $j(z_k), \beta, y_1, \ldots, y_{m-2}, \alpha$ are pairwise distinct and
		\[(j(z_k) - \alpha)^{e_1} (\beta - \alpha)^{e_2}\prod_{i=1}^{m-2} (y_i - \alpha)^{e_{i+2}} =1\]
		for some $e_1, \ldots, e_m \in \Z \setminus \{0\}$. Further, the tuples that arise in this way for different $k$ are all distinct, since the $j(z_k)$ are all distinct.
		
		By the pigeonhole principle, there is thus some $m \in \{2, \ldots, 2n\}$ for which there exist infinitely many $m$-tuples $(w_1, \ldots, w_m)$ of singular moduli such that $w_1, \ldots, w_m, \alpha$ are pairwise distinct and 
		\[\prod_{i=1}^m (w_i - \alpha)^{b_i} = 1\]
		for some $b_1, \ldots, b_m \in \Z \setminus \{0\}$. This contradicts Theorem~\ref{thm:fixedy} and so we are done.
	\end{proof}
	
	Hence a multiplicative special curve as in \eqref{eq:MSC} must have $f$ non-constant.

	\begin{prop}\label{prop:multindnoconst}
		Let $n \in \Z_{>0}$. Let $f_1, \ldots, f_n, f$ be pairwise distinct, non-constant $j$-maps. Then the functions $h_1, \ldots, h_n \colon \h \to \C$ defined by $h_i(z) = f_i(z) - f(z)$ are multiplicatively independent modulo constants. 
	\end{prop}
	
	\begin{proof}
		We will find some $z \in \h$ where precisely one of the functions $h_i$ vanishes. By a change of variables, we may assume that $f(z) = j(z)$ and $f_i(z) = j(r_i z + s_i)$ for some $r_i, s_i \in \Q$ such that $r_i > 0$ and $0 \leq s_i < 1$. Note that $(r_i, s_i) \neq (1, 0)$ since $f_i \neq f$. We may and do assume that the pairs $(r_i, s_i)$ are strictly increasing when ordered lexicographically.
		
		Suppose first that $r_1 \geq 1$. Let 
		\[ z_0 = - \frac{s_1}{2 r_1} + \frac{\sqrt{4r_1 - s_1^2}}{2 r_1} i,\]
		so that
		\[ \frac{-1}{r_1 z_0 + s_1} = z_0.\]
		Observe that $\lvert r_1 z_0 + s_1 \rvert = \sqrt{r_1} \geq 1$. If $r_1 > 1$, then $r_1 z_0 + s_1 \in \mathfrak{F}_j$. If $r_1 = 1$, then $s_1 > 0$ and $z_0$ is on the left hand side of the lower boundary of $\mathfrak{F}_j$ and $r_1 z_0 + s_1$ is the reflection of $z_0$ in the imaginary axis.
		
		Since $j(-1/z) = j(z)$, we have that $f_1(z_0) = f(z_0)$. If $r_i > r_1$, then, by Proposition~\ref{prop:fundom}, $\im (r_i z_0 + s_i) > \im (r_1 z_0 + s_1)$ and hence $j(r_i z_0 + s_i) \neq j(r_1 z_0 + s_1)$. If $r_i = r_1$ for $i \geq 2$, then $s_1 < s_i < 1$ and $j(r_i z_0 + s_i) \neq j(r_1 z_0 + s_1)$ by Proposition~\ref{prop:fundom} again. Thus, $f_i(z_0) = f(z_0)$ if and only if $i=1$ and we are done.
		
		Now suppose that $r_1 < 1$. Let $k \in \Z$ be such that $0 \leq kr_1 - s_1 < r_1$. Let 
		\[ z_1 = -\frac{k}{2} - \frac{s_1}{2 r_1} + \frac{\sqrt{4 r_1 - (k r_1 - s_1)^2}}{2 r_1} i,\]
		so that
		\[\frac{-1}{z_1+k} = r_1 z_1 + s_1.\]
		Hence, $f_1(z_1) = f(z_1)$. Observe also that $\lvert z_1 + k \rvert = 1/\sqrt{r_1}$ and $z_1 + k \in \mathfrak{F}_j \setminus \partial \mathfrak{F}_j$ and so $r_1 z_1 + s_1 \in S \mathfrak{F}_j \setminus \partial(S \mathfrak{F}_j)$, where $S$ denotes the transformation $z \mapsto -1/z$. Thus, by Proposition~\ref{prop:fundom}, the points in the $\SL$-orbit of $z_1$ with imaginary part $\geq \im (r_1 z_1 + s_1)$ are the elements of
		\[ \{z_1 + m : m \in \Z\} \cup \{r_1 z_1 + s_1 + l : l \in \Z\}.\]
		In particular, $r_i z_1 + s_i$ is not in the $\SL$-orbit of $z_1$ if $i \geq 2$, since $(r_i, s_i) \neq (1, 0)$. Hence, $f_i(z_1) = f(z_1)$ if and only if $i=1$. The proof is thus complete. 
	\end{proof}
	
	Therefore a multiplicative special curve as in \eqref{eq:MSC} must have some $f_i$ constant.
	
	\subsection{The shape of multiplicative special curves}\label{subsec:shape}
	
	\begin{proof}[Proof of Theorem~\ref{thm:multind}]
		Let $n \in \Z_{>0}$. Let $f_1, \ldots, f_n, f$ be pairwise distinct $j$-maps, at least one of which is non-constant. Suppose that $a_1, \ldots, a_n \in \Z \setminus \{0\}$ and $c \in \C^\times$ are such that
		\begin{align}\label{eq:multpf}
			\prod_{i=1}^n (f_i(z) - f(z))^{a_i} = c
		\end{align}
		for all $z \in \h$.
		
		By Proposition~\ref{prop:multindconstsing}, the $j$-map $f$ must be non-constant. Thus, by Proposition~\ref{prop:multind1map}, at least one of the $j$-maps $f_1, \ldots, f_n$ must be non-constant. By Proposition~\ref{prop:multindnoconst}, at least one of the $j$-maps $f_1, \ldots, f_n$ is constant. After relabelling, we thus have that
		\[ \prod_{i \in I_1} (f_i -f)^{a_i} \prod_{i \in I_2} (f_i - f)^{a_i} = c\]
		for all $z \in \h$, where the $j$-map $f$ is non-constant and $I_1, I_2$ are non-empty index sets such that $I_1 \cup I_2 = \{1, \ldots, n\}$ and the $j$-map $f_i$ is constant if $i \in I_1$ and non-constant if $i \in I_2$.
		
		By a change of variables, we may write $f(z) = j(z)$. For $i \in I_1$, let $\alpha_i$ be the singular modulus such that $f_i = \alpha_i$. Note that the $\alpha_i$ must be pairwise distinct, since the $f_i$ are. For $i \in I_2$, there is, by Proposition~\ref{prop:unique}, a unique $N_i \in \Z_{>0}$ and $g_i \in C(N_i)$ such that $f_i(z) = j(g_i z)$ and $N_i > 1$ since $f_i(z) \neq j(z)$. Rearrange to obtain that
		\begin{align}\label{eq:multpf2}
			c' \prod_{i \in I_1} (j(z) - \alpha_i)^{a_i} = \prod_{i \in I_2} (j(z) - j(g_i z))^{-a_i}
		\end{align}
		for all $z \in \h$, where
		\[ c' = \frac{(-1)^{a_1 + \ldots + a_n}}{c}.\] 
		We will show that the right hand side of \eqref{eq:multpf2} must be a product of powers of functions
		\[\prod_{g \in C(N_i)} (j(z) - j(gz)).\] 
		
		Rewrite the right hand side of \eqref{eq:multpf2} by grouping factors with the same $N_i$ to obtain that
		\begin{align}\label{eq:multpf3}
			c' \prod_{i \in I_1} (j(z) - \alpha_i)^{a_i} = \prod_{i \in I_3} \prod_{g \in S_i} (j(z) - j(g z))^{a_i(g)}
		\end{align}
		for all $z \in \h$, where $I_3$ is a new index set and, for each $i \in I_3$, $S_i \subset C(M_i)$ is non-empty, the $M_i \in \Z_{>1}$ are pairwise distinct, and the $a_i(g)$ belong to $\Z \setminus \{0\}$. We will show that, for each $i \in I_3$, we have that $S_i = C(M_i)$ and the $a_i(g)$ are equal for every $g \in C(M_i)$.
		
		Suppose then that there is $i_0 \in I_3$ with the property that there exist $g_0 \in S_{i_0}$ and $h_0 \in C(M_{i_0})$ such that either $h_0 \notin S_{i_0}$ or $h_0 \in S_{i_0}$ but $a_{i_0}(h_0) \neq a_{i_0}(g_0)$. By Proposition~\ref{prop:jmaps}, there exists $\gamma \in \SL$ such that $j(g_0 \gamma z) = j(h_0 z)$.
		
		Since the function $j(z)$ is invariant under the map $z \mapsto \gamma z$, we obtain from \eqref{eq:multpf3} that
		\begin{align}\label{eq:multpf4}
			\prod_{i \in I_3} \prod_{g \in S_i} (j(z) - j(g z))^{a_i(g)} = \prod_{i \in I_3} \prod_{g \in S_i } (j(z) - j(g \gamma z))^{a_i(g)}
		\end{align}
		for all $z \in \h$. Now, by Proposition~\ref{prop:jmaps}, the factor $j(z) - j(h_0 z)$ appears on the right hand side of \eqref{eq:multpf4} with exponent $a_{i_0}(g_0)$, and either does not appear on the left hand side (if $h_0 \notin S_{i_0}$) or appears on the left hand side with exponent equal to $a_{i_0}(h_0)$, which is not equal to $a_{i_0}(g_0)$, otherwise.
		
		The equation \eqref{eq:multpf4} thus implies that there exist $l \in \Z_{>0}$ and non-constant $j$-maps $f_1, \ldots, f_l, f$ with $f(z)=j(z)$ and $f_1(z) = j(h_0 z)$ such that the functions $v_i$ for $i=1, \ldots, l$ defined by $v_i(z) = f_i(z) - f(z)$ are multiplicatively dependent modulo constants. This though contradicts Proposition~\ref{prop:multindnoconst}.
		
		Therefore, in \eqref{eq:multpf3}, we must have, for each $i \in I_3$, that $S_i = C(M_i)$ and that the $a_i(g)$ are equal for every $g \in C(M_i)$. The right hand side of \eqref{eq:multpf3} may thus be rewritten to obtain that
		\begin{align}\label{eq:multpf5}
			c' \prod_{i \in I_1} (j(z) - \alpha_i)^{a_i} = \prod_{i \in I_3} \prod_{g \in C(M_i)} (j(z)-j(gz))^{b_i}
		\end{align}
		for all $z \in \h$, for some $b_i \in \Z \setminus \{0\}$, and pairwise distinct $M_i \in \Z_{>1}$. The right hand side is thus equal to the function
		\[\prod_{i \in I_3} F_{M_i}(j(z))^{b_i},\]
		the zeros and poles of which are thus equal to the $\alpha_i$ on the left hand side of \eqref{eq:multpf5}.
	\end{proof}
	
	\subsection{Determining the multiplicative special curves}
	
	We now complete the proof of Theorem~\ref{thm:MSC}. Let $n \in \Z_{>0}$. We will show that there are only finitely many multiplicative special curves in $\C^{n+1}$ and these may be determined effectively.
	
	\begin{proof}[Proof of Theorem~\ref{thm:MSC}]
		Suppose that
		\[T = \{(f_1(z), \ldots, f_n(z), f(z)) : z \in \h\}\]
		is a multiplicative special curve in $\C^{n+1}$. Then, by Theorem~\ref{thm:MSCshape}, we may reorder the first $n$ coordinates of $T$ in such a way that
		\[ T = \{(\alpha_1, \ldots, \alpha_m, j(g_1 z), \ldots, j(g_l z), j(z)) : z \in \h\},\]
		where 
		\begin{enumerate}
			\item $\alpha_1, \ldots, \alpha_m$ are pairwise distinct and such that 
			\[ \{\alpha_1, \ldots, \alpha_m\} = \{\alpha \in \C : \alpha \mbox{ is either a zero or a pole of } \prod_{i=1}^k F_{N_i}(X)^{b_i} \};\]
			\item $g_1, \ldots, g_l \in \GL$ are pairwise distinct and such that
			\[ \{g_1, \ldots, g_l\} = \bigcup_{i=1}^k C(N_i); \]
		\end{enumerate}
		for some $k \in \Z_{>0}$, $b_1, \ldots, b_k \in \Z \setminus \{0\}$, and pairwise distinct $N_1, \ldots, N_k \in \Z_{>1}$.
		In particular,
		\[ m+l = n.\] 
		Also,
		\[ l = \sum_{i=1}^k \# C(N_i).\]
		Since (\cite[p.~53]{Lang87})
		\[ \# C(N_i) = N_i \prod_{p \mid N_i} \Big(1 + \frac{1}{p}\Big),\]
		we have that $\#C(N_i) \geq N_i+1$.
		
		Corollary~\ref{cor:vanish} implies that
		\[\prod_{i=1}^k F_{N_i}(X)^{b_i}\] 
		is non-constant. Hence, $m \geq 1$. Thus we must have that
		\[\sum_{i=1}^k \# C(N_i) \leq n-1.\]
		So $\max \{N_1, \ldots, N_k\} \leq n-2$. Since $N_1, \ldots, N_k$ are pairwise distinct and $\geq 2$, we must have that
		\[\sum_{i=2}^{k+1} (i+1) \leq n-1.\]
		Thus
		\[ \frac{1}{2} k(k+5) \leq n-1,\]
		and hence
		\[ k \leq \frac{1}{2} (\sqrt{8n+17} - 5).\]
		In particular, there are only finitely many possibilities for $k, N_1, \ldots, N_k$ and these may be computed.
		
		Let $k, N_1, \ldots, N_k$ be such a possible choice for a multiplicative special curve in $\C^{n+1}$. Compute
		\[ l = \sum_{i=1}^k \# C(N_i).\]
		The corresponding polynomials $F_{N_i}$ may also be computed \cite[\S13B]{Cox89}. Let $\beta_1, \ldots, \beta_r$ be pairwise distinct and such that
		\[ \{\beta_1, \ldots, \beta_r\} = \{ \beta \in \C : F_{N_i}(\beta) = 0 \mbox{ for some } i=1,\ldots, k\}.\]
		Write $e_{i, u}$ for the multiplicity of $\beta_u$ as a root of $F_{N_i}$. Let $d_i$ be the leading coefficient of $F_{N_i}$. Note that $d_i \in \Z \setminus \{0\}$. Let $p_1, \ldots, p_t$ be a complete list of the prime factors of $d_1, \ldots, d_k$. Let $f_{i, v}$ be the exponent of $p_v$ occurring in the prime factorisation of $d_i$.
		
		The choice $k, N_1, \ldots, N_k$ then gives rise to a multiplicative special curve in $\C^{n+1}$ if and only if there exist $b_1, \ldots, b_k \in \Z \setminus \{0\}$ such that
		\[\sum_{i=1}^k b_i f_{i, v} = 0 \]
		for every $v \in \{1, \ldots, t\}$ and
		\[ \sum_{i=1}^k b_i e_{i, u} = 0\]
		for exactly $n-l$ choices of $u \in \{1, \ldots, r\}$. This condition may be checked effectively. Consequently, there are only finitely many multiplicative special curves in $\C^{n+1}$ and these may be determined effectively.
		
		Now suppose that $n \leq 5$. Then
		\[k \leq \frac{1}{2}(\sqrt{57} -5) < \frac{3}{2}.\]
		So $k=1$ is the only possibility. And
		\[N_1 \leq 3.\]
		So the only possible multiplicative special curves in $\C^{n+1}$ arise with $k =1 $ and $N_1 \in \{2, 3\}$. If $N_1 = 2$, then $l=3$ and so one needs $m \leq 2$, which is impossible since $F_2$ has three distinct roots (see Example~\ref{eg:mod}). If $N_1 = 3$, then $l=4$ and so one needs $m \leq 1$, but the polynomial
		\[F_3(X) = -X(X-8000)^2(X+32768)^2(X-54000)\]
		has four distinct roots. Thus, there are no multiplicative special curves in $\C^{n+1}$ for $n \in \{1, \ldots , 5\}$. 
	\end{proof}
	
	Finally, we remark that there does exist a multiplicative special curve in $\C^7$, namely that given in Example~\ref{eg:mod}.
	
	\section{Weakly special subvarieties and Ax--Lindemann}\label{sec:Ax}
	
	\subsection{Weakly special subvarieties}\label{subsec:wss}
	
	For the proof of Theorem~\ref{thm:main}, we will need the notion of (weakly) special subvarieties. Varieties and subvarieties are always irreducible over $\C$.
	
	\begin{definition}\label{def:specials}
		Let $m, n \in \Z_{>0}$.
		\begin{enumerate}
			\item A weakly special subvariety of $\C^m$ is an irreducible component of a subvariety of $\C^m$ defined by equations of the form $\Phi_N(x_i, x_k)=0$ and $x_l = c$ for $N \in \Z_{>0}$ and $c \in \C$.
			\item A special point of $\C^m$ is a point $(x_1, \ldots, x_m) \in \C^m$ such that $x_1, \ldots, x_m$ are singular moduli.
			\item A special subvariety of $\C^m$ is a weakly special subvariety of $\C^m$ which contains a special point of $\C^m$. Equivalently, a weakly special subvariety for which any constant coordinates are singular moduli.
			\item A weakly special subvariety of $(\C^\times)^n$ is a coset of a subtorus (i.e. a coset of an irreducible algebraic subgroup of $(\C^\times)^n$).
			\item A special point of $(\C^\times)^n$ is a point $(\zeta_1, \ldots, \zeta_n) \in (\C^\times)^n$ such that $\zeta_1, \ldots, \zeta_n$ are roots of unity.
			\item A special subvariety of $(\C^\times)^n$ is a weakly special subvariety of $(\C^\times)^n$ which contains a special point of $(\C^\times)^n$.
			\item A (weakly) special subvariety of $\C^m \times (\C^\times)^n$ is a product $M \times T$, where $M \subset \C^m$ is a (weakly) special subvariety of $\C^m$ and $T \subset (\C^\times)^n$ is a (weakly) special subvariety of $(\C^\times)^n$.
		\end{enumerate}
	\end{definition}
	
	It follows from this definition that a weakly special subvariety $T \subset (\C^\times)^n$ is defined by equations of the form
	\[t_1^{a_1} \cdots t_n^{a_n} = c\]
	for some $c \in \C^\times$ and $a_1, \ldots, a_n \in \Z$ not all zero. Also, $T$ is a special subvariety if and only if $T$ may be defined by equations of this kind with the additional property that every such $c$ is a root of unity. See, for example, \cite[Remark~1.0.1]{Zannier12}.
	
	It follows from the above description that special subvarieties of $\C^m$ and $(\C^\times)^n$ are defined over $\alg$. Special subvarieties of $\C^m$ have the following useful properties.
	
	\begin{prop}[{\cite[Proposition~2.1]{BiluLucaMasser17}}]\label{prop:specialsdense}
		Let $m \in \Z_{>0}$. Let $M \subset \C^m$ be a positive-dimensional special subvariety. Then $M$ contains a Zariski-dense union of special subvarieties of $\C^m$ of dimension $1$.
	\end{prop}
	
	\begin{prop}[{\cite[Proposition~2.3]{BiluLucaMasser17}}]\label{prop:jset}
		Let $m \in \Z_{>0}$. Let $M \subset \C^m$. Then $M$ is a special subvariety of dimension $1$ if and only if there exist $j$-maps $f_1, \ldots, f_n$, at least one of which is non-constant, such that
		\[M = \{(f_1(z), \ldots, f_n(z)) : z \in \h\}.\]
	\end{prop}
	
	In particular, a multiplicative special curve in $\C^{n+1}$ is a special subvariety of $\C^{n+1}$ of dimension $1$.
	
	\subsection{Ax--Lindemann}\label{subsec:Ax}
	
	Now we come to the functional transcendence result of Pila \cite{Pila11} which we will apply in the proof of Theorem~\ref{thm:main}. 
	
	For $m, n \in \Z_{>0}$, let
	\[ X = \C^m \times (\C^\times)^n\]
	and
	\[ U = \h^m \times \C^n.\]
	Define $e \colon \C \to \C^\times$ by $e(t) = \exp(2 \pi i t)$. Define $\pi \colon U \to X$ by
	\[\pi(z_1, \ldots, z_m, t_1, \ldots, t_n) = (j(z_1), \ldots, j(z_m), e(t_1), \ldots, e(t_n)).\]
	
	We make the following definitions.
	
	\begin{definition}[{\cite[Definition~6.1]{Pila11}}]\label{def:caccmpnt}
		Let $Z \subset U$ be a complex analytic subset. A complex algebraic component of $Z$ is a positive-dimensional connected component $Y \subset W \cap U$ for some algebraic subvariety $W \subset \C^{m+n}$ such that $Y \subset Z$. Here, $W \cap U$ is considered as a complex analytic set. A maximal complex algebraic component of $Z$ is a complex algebraic component of $Z$ which is not contained in any complex algebraic component of $Z$ of strictly larger dimension.
	\end{definition}
	
	\begin{definition}[{\cite[Definition~6.5]{Pila11}}]\label{def:prespecials}
		Let $m, n \in \Z_{>0}$.
		\begin{enumerate}
			\item A weakly special subvariety of $\h^m$ is (the intersection with $\h^m$ of) a subvariety defined by equations of the forms $z_i = g_{i, k} z_k$ and $z_l = c_l$ for some matrices $g_{i, k} \in \GL$ and constants $c_l \in \h$.
			\item A weakly special subvariety for $e$ of $\C^n$ is a subvariety of the form $b+L$ for some $b \in \C^n$ and linear subspace $L \subset \C^n$ defined over $\Q$.
			\item A weakly special subvariety of $U = \h^m \times \C^n$ is a product $M \times T$, where $M \subset \h^m$ is a weakly special subvariety of $\h^m$ and $T \subset \C^n$ is a weakly special subvariety for $e$ of $\C^n$.
		\end{enumerate}
	\end{definition}
	
	In particular, if $W$ is a weakly special subvariety of $U$, then $\pi(W)$ is a weakly special subvariety of $X$. The slightly cumbersome terminology of weakly special subvariety for $e$ of $\C^n$ is chosen to avoid confusion with the definition of a weakly special subvariety of $\C^n$ in Definition~\ref{def:specials}(1).
	
	The functional transcendence result we require is the following statement, which Pila calls an Ax--Lindemann result. Note that Pila \cite{Pila11} formulates his result with the ordinary complex exponential function $\exp \colon \C \to \C^\times$ in place of the function $e$, but this difference is of no consequence for our purposes. 
	
	\begin{thm}[{\cite[Theorem~6.8]{Pila11}}]\label{thm:AxL}
		Let $V \subset X$ be an algebraic subvariety. If $Y$ is a maximal complex algebraic component of $\pi^{-1}(V)$, then $Y$ is a weakly special subvariety of $U$.
	\end{thm}

	\section{The proof of Theorem~\ref{thm:main}}\label{sec:pf}
	
	We will prove Theorem~\ref{thm:main} by applying the so-called Pila--Zannier strategy of o-minimal point counting. This strategy was proposed by Zannier and was first used by Pila and Zannier \cite{PilaZannier08} to give a new proof of the Manin--Mumford conjecture. The approach used here is similar to that employed in \cite{Fowler21, PilaTsimerman17}. For background on o-minimality and on the Pila--Zannier method, see Pila's book \cite{Pila22}. 
	
	\subsection{The counting theorem for semirational points}\label{subsec:count}
	
	We will use an extension, due to Habegger and Pila \cite[Corollary~7.2]{HabeggerPila16}, of the Pila--Wilkie o-minimal counting theorem \cite{PilaWilkie06}. We will always work in the o-minimal structure $\RAE$; see \cite[p.~77]{Pila22} for details of this structure. Definable will mean definable with parameters in $\RAE$. Complex numbers, when considered as elements of definable sets, will be identified with their real and imaginary parts. Throughout this section, constants $c=c(\ldots)$ will be positive and have only the indicated dependencies.
	
	To state Habegger and Pila's result, we need to define the $k$-height of a real number. Let $k \in \Z_{>0}$. For $y \in \R$, define the $k$-height of $y$ by 
	\begin{align*}
		H_k(y) = \min \{ &\max \{\lvert a_0 \rvert, \ldots, \lvert a_k \rvert\} : a_0, \ldots, a_k \mbox{ are coprime integers, not all}\\ 
		&\mbox{zero, such that } a_k y^k + \ldots + a_0 = 0\},
	\end{align*}
	with the convention that $\min \emptyset = \infty$. Note that $y \in \R$ thus has $H_k(y) < \infty$ if and only if $[\Q(y) : \Q] \leq k$. For $y=(y_1, \ldots, y_n) \in \R^n$, define 
	\[H_k(y) = \max \{ H_k(y_1), \ldots, H_k(y_n) \}.\]
	The $k$-height is related to the multiplicative height in the following way.
	
	\begin{prop}\label{prop:hts}
		Let $d \in \Z_{>0}$. There exists a constant $c(d) > 0$ with the property that if $\alpha \in \alg$ is such that $[\Q(\alpha) : \Q] = d$, then
		\[H_d(\alpha) \leq c(d) H(\alpha)^d.\]
	\end{prop}
	
	\begin{proof}
		Let 
		\[c(d) = {d \choose \lfloor d/2 \rfloor}.\]
		Suppose that
		\[f(t) = a_d t^d + \ldots + a_0\]
		is a minimal polynomial over $\Z$ of some $\alpha \in \alg$. Then
		\[		H_d(\alpha) \leq \max \{\lvert a_0 \rvert, \ldots, \lvert a_d \rvert \}
		\leq c(d) M(f)
		= c(d) H(\alpha)^d,\]
		where the second inequality is \cite[Lemma~1.6.7]{BombieriGubler06} and the final equality is \cite[Proposition~1.6.6]{BombieriGubler06}. Here $M(f)$ denotes the Mahler measure of $f$.
	\end{proof}
	
	Habegger and Pila's point counting result is the following.

	\begin{thm}[{\cite[Corollary~7.2]{HabeggerPila16}}]\label{thm:count}
		Let $F \subset \R^l \times \R^m \times \R^n$ be a definable family parametrised by $\R^l$. Let $\epsilon > 0$ and $k \in \Z_{>0}$. Let $\pi_1 \colon \R^m \times \R^n \to \R^m$ and $\pi_2 \colon \R^m \times \R^n \to \R^n$ be the projection maps. Then there exists a constant $c = c(F, k, \epsilon) > 0$ with the following property.
		
		Let $x \in \R^l$. Write $F_x \subset \R^m \times \R^n$ for the fibre of $F$ over $x$. If $T \geq 1$ and 
		\[\Sigma \subset \{(y, z) \in F_x : H_k(y) \leq T\}\]
		is such that $\# \pi_2(\Sigma) > cT^\epsilon$, then there exists a continuous, definable function $\beta \colon [0, 1] \to F_x$ such that:
		\begin{enumerate}
			\item The composition $\pi_1 \circ \beta \colon [0, 1] \to \R^m$ is semialgebraic and its restriction to $(0, 1)$ is real analytic.
			\item The composition $\pi_2 \circ \beta \colon [0, 1] \to \R^n$ is non-constant.
			\item $\pi_2(\beta(0)) \in \pi_2(\Sigma)$.
			\item The restriction of $\beta$ to $(0, 1)$ is real analytic.
		\end{enumerate}
	\end{thm}
	
	The constant $c = c(F, k, \epsilon)$ here is not effective. For (4), we use the fact that $\RAE$ admits analytic cell decomposition \cite[Theorem~8.8]{vandenDriesMiller94}.
	
	We will also require the following bound on the size of the exponents in a multiplicative dependency.
	
	\begin{prop}\label{prop:nonminexpbd}
		Let $n \in \Z_{>0}$. There exist constants $c_1(n), c_2(n)>0$ with the following property. Let $L$ be a number field and $d= [L : \Q]$. If $\alpha_1, \ldots, \alpha_n \in L^\times$ are pairwise distinct and such that
		\[ \prod_{i=1}^n \alpha_i^{a_i} = 1\]
		for some $a_1, \ldots, a_n \in \Z \setminus \{0\}$, then there exist $b_1, \ldots, b_n \in \Z \setminus \{0\}$ such that
		\[ \prod_{i=1}^n \alpha_i^{b_i} = 1\]
		and, for every $i \in \{1, \ldots, n\}$,
		\[\lvert b_i \rvert \leq c_1(n) (d \max \{1, h(\alpha_1), \ldots, h(\alpha_n)\})^{c_2(n)}.\]
	\end{prop}
	
	This will follow from the following bound, which covers the case where the multiplicative dependency is minimal.
	
	\begin{lem}\label{lem:expbd}
		Let $n \in \Z_{>0}$. There exists an explicit constant $c(n)>0$ with the following property. Let $L$ be a number field and $d= [L : \Q]$. If $\alpha_1, \ldots, \alpha_n \in L^\times$ are pairwise distinct and the set $\{\alpha_1, \ldots, \alpha_n\}$ is minimally multiplicatively dependent, then there exist $b_1, \ldots, b_n \in \Z \setminus \{0\}$ such that
		\[ \prod_{i=1}^n \alpha_i^{b_i} = 1\]
		and
		\[\lvert b_i \rvert \leq c(n) d^{n+1} (1+ \log d) \prod_{\substack{k=1\\ k \neq i}}^n h(\alpha_k).\]
	\end{lem}
	
	\begin{proof}
		If $n=1$, then $\alpha_1$ is a root of unity of degree $\leq d$. Hence, $\alpha_1$ is a primitive $N$th root of unity for some $N$ with $\phi(N) \leq d$, where $\phi$ denotes Euler's totient function. The desired result then follows from the elementary bound 
		\[\phi(N) \geq \sqrt{\frac{N}{2}}.\]

		For $n \geq 2$, this is a result of Yu \cite[Corollary~3.2]{LoherMasser04}. The version stated in \cite{LoherMasser04} has $d^n \log d$ in place of the $d^{n+1} (1 + \log d)$ here; the slight weakening here allows one to state a uniform result for all $d, n \geq 1$ which still suffices for our purposes.
	\end{proof}
	
	Proposition~\ref{prop:nonminexpbd} follows from Lemma~\ref{lem:expbd} via the following elementary lemma.
	
	\begin{lem}\label{lem:vect}
		Let $n \in \Z_{>1}$. Let $v, w \in \Z^n$. Suppose that, for some $k \in \{1, \ldots, n-1\}$, we have that
		\[ v = (v_1, \ldots, v_k, 0, \ldots, 0),\]
		where $v_1, \ldots, v_k \neq 0$, and that
		\[ w = (w_1, \ldots, w_n)\]
		with $w_{k+1} \neq 0$. Let 
		\[\lambda = 1+ \max\{ \lvert v_1 \rvert, \ldots, \lvert v_k \rvert, \lvert w_1 \rvert, \ldots, \lvert w_n \rvert \}.\]
		Let
		\[ u = v + \lambda w,\]
		and write 
		\[ u = (u_1, \ldots, u_n). \]
		Then $u_1, \ldots, u_{k+1} \neq 0$ and
		\[ \lvert u_i \rvert \leq 2 \lambda^2 \mbox{ for } i=1, \ldots, n.\]
	\end{lem}
	
	\begin{proof}
		For $i=1, \ldots, k$, if $w_i \neq 0$, then $\lvert w_i \rvert \geq 1$ and so $\lambda \lvert w_i \rvert > \lvert v_i \rvert$. The result follows immediately since $m^2 + m \leq 2 m^2$ for all $m \in \Z$.
	\end{proof}

	\begin{proof}[Proof of Proposition~\ref{prop:nonminexpbd}]
		Let $L$ be a number field and $d= [L : \Q]$. Suppose that $\alpha_1, \ldots, \alpha_n \in L^\times$ are pairwise distinct and such that
		\[ \prod_{i=1}^n \alpha_i^{a_i} = 1\]
		for some $a_1, \ldots, a_n \in \Z \setminus \{0\}$. The set
		\[ \mathcal{S}= \{\alpha_1, \ldots, \alpha_n\}\]
		is thus multiplicatively dependent, but not necessarily minimally multiplicatively dependent. For each $i=1,\ldots, n$ though, there exists $S_i \subset \mathcal{S}$ such that $\alpha_i \in S_i$ and $S_i$ is minimally multiplicatively dependent, see e.g. \cite[Lemma~5.9]{Fowler21}. We will apply Lemma~\ref{lem:expbd} to each set $S_i$.
		
		By Lemma~\ref{lem:expbd}, there exist constants $c_1(n), c_2(n)>0$ such that, for each $i=1, \ldots, n$, there exist $b_{i, k} \in \Z \setminus \{0\}$ for $k \in S_i$ with
		\[ \prod_{k \in S_i} \alpha_k^{b_{i, k}} = 1\]
		and 
		\[ \lvert b_{i, k} \rvert \leq c_1(n) (d \max \{1, h(\alpha_1), \ldots, h(\alpha_n)\})^{c_2(n)}. \]
		Now let $v_i \in \Z^n$ be the vector with $k$th coordinate $v_{i, k}$ equal to $b_{i, k}$ if $k \in S_i$ and $0$ otherwise. Hence, for $i=1, \ldots, n$, we have that $v_{i, i} \neq 0$ and
		\[ \prod_{k=1}^n \alpha_k^{v_{i, k}} = 1.\]
		
		Let 
		\[\mu = 1 + \max \{ \lvert v_{i, k} \rvert : i, k \in \{1, \ldots, n\}\}.\]
		Apply Lemma~\ref{lem:vect} inductively to $v_1, \ldots, v_n$ to obtain a vector
		\[w = (w_1, \ldots, w_n) \in \Z^n\] 
		which is a $\Z$-linear combination of $v_1, \ldots, v_n$ and such that $w_1, \ldots, w_n \neq 0$ and
		\[ \lvert w_i \rvert \leq c_3(n) \mu^{c_4(n)} \mbox{ for } i=1, \ldots, n\]
		for some constants $c_3(n), c_4(n)>0$. In particular, there are constants $c_5(n), c_6(n)>0$ such that
		\[ \lvert w_i \rvert \leq c_5(n)  (d \max \{1, h(\alpha_1), \ldots, h(\alpha_n)\})^{c_6(n)} \mbox{ for } i=1, \ldots, n.\]
		Since $w$ is a $\Z$-linear combination of $v_1, \ldots, v_n$, one has that
		\[\prod_{i=1}^n \alpha_i^{w_i} = 1,\]
		as required.
	\end{proof}
	
	\subsection{Completing the proof of Theorem~\ref{thm:main}}\label{subsec:pfmain}
	
	Now we come to the proof of Theorem~\ref{thm:main}. Fix $n \in \Z_{>0}$. In the proof, $c_1, c_2, \ldots$ will denote positive constants which depend only on $n$. Any other dependencies among constants will be explicitly indicated.
	
	By Theorem~\ref{thm:MSC}, there are only finitely many multiplicative special curves in $\C^{n+1}$. Since every multiplicative special curve is defined over $\alg$, we may fix some number field $K$ over which all the multiplicative special curves in $\C^{n+1}$ are defined.
	
	Define the complexity $\Delta$ of an $(n+1)$-tuple $(x_1, \ldots, x_n, y)$ of singular moduli $x_1, \ldots, x_n, y$ by 
	\[\Delta = \max \{\lvert \Delta(x_1) \rvert, \ldots, \lvert \Delta(x_n) \rvert, \lvert \Delta(y) \rvert\}.\]
	Recall that $e \colon \C \to \C^\times$ is given by $e(z) = \exp(2 \pi i z)$. Let $\mathfrak{F}_e = \{ z \in \C : 0 \leq \re z < 1\}$, so that $e$ restricted to $\mathfrak{F}_e$ is a bijection. 
	
	Recall that by definable we always mean definable (with parameters) in the structure $\RAE$. For the purposes of definability, we identify subsets of $\C^n$ with subsets of $\R^{2n}$ in the obvious way. In particular, the usual field operations on $\C$ are thus definable. We will use the fact that the restricted functions $e \colon \mathfrak{F}_e \to \C^\times$ and $j \colon \mathfrak{F}_j \to \C$ are both definable. The restriction of $e$ is definable using restricted $\sin$ and $\cos$ and the unrestricted real exponential function. The definability of $j$ may then be deduced from its $q$-expansion, see \cite[Example~4.14]{Zannier12} for the details.
	
	Let
	\begin{align*}
		Y = \Big\{&(z_1, \ldots, z_n, z, w_1, \ldots, w_n, w, u_1, \ldots, u_n, r_1, \ldots, r_n, s)\\
		&\in \mathfrak{F}_j^{2(n+1)} \times \mathfrak{F}_e^n \times \R^{n+1} : \sum_{i=1}^n r_i u_i = s, w=z, \mbox{ and } \\
		& w_i = z_i \mbox{ and } e(u_i) = j(z_i) - j(z)\mbox{ for } i=1,\ldots, n \Big\}
	\end{align*}
	and
	\begin{align*} 
		Z = \Big\{&(z_1, \ldots, z_n, z, r_1, \ldots, r_n, s) \in \mathfrak{F}_j^{n+1} \times \R^{n+1} :\\ 
		&\exists (u_1, \ldots, u_n) \in \mathfrak{F}_e^n \mbox{ such that }\\
		&(z_1, \ldots, z_n, z, z_1, \ldots, z_n, z, u_1, \ldots, u_n, r_1, \ldots, r_n, s) \in Y\Big\}.
	\end{align*}
	The sets $Y, Z$ are both definable.
	
	Suppose that $(x_1, \ldots, x_n, y)$ is an $(n+1)$-tuple of pairwise distinct singular moduli $x_1, \ldots, x_n, y$ such that
	\[ \prod_{i=1}^n (x_i - y)^{a_i} = 1\]
	for some $a_i \in \Z \setminus \{0\}$. Let $\Delta$ be the complexity of this tuple. Let
	\[d = [\Q(x_1, \ldots, x_n, y) : \Q].\]

	By Proposition~\ref{prop:nonminexpbd}, we may assume that
	\[ \lvert a_i \rvert \leq c_1 (d \max\{1, h(x_1), \ldots, h(x_n), h(y)\})^{c_2}\]
	for $i \in \{1, \ldots, n\}$. Then apply Proposition~\ref{prop:galbd} (with $\epsilon=1/4$ say) to give an upper bound on $d$ and Proposition~\ref{prop:htsingmod} to bound the logarithmic heights of the singular moduli. One thereby obtains that, for $i \in \{1, \ldots, n\}$, 
	\[\lvert a_i \rvert \leq c_3 \Delta^{c_4}.\]
	
	Let
	\[ (\tau_1, \ldots, \tau_n, \tau, \nu_1, \ldots, \nu_n) \in \mathfrak{F}_j^{n+1} \times \mathfrak{F}_e^n\]
	be the preimage of
	\[ (x_1, \ldots, x_n, y, x_1-y, \ldots, x_n - y)\]
	with respect to the map $(j, e) \colon \mathfrak{F}_j^{n+1} \times \mathfrak{F}_e^n \to \C^{n+1} \times (\C^\times)^n$. Note that $\tau_1, \ldots, \tau_n, \tau$ are all quadratic, since they are the preimages for $j$ of singular moduli. By Proposition~\ref{prop:htpresingmod}, the real and imaginary parts of $\tau_1, \ldots, \tau_n, \tau$ all have multiplicative height $\leq 4 \Delta/3$. Observe also that 
	\[ \sum_{i=1}^n a_i \nu_i \in \Z,\]
	since
	\[ \prod_{i=1}^n e(\nu_i)^{a_i} = 1.\]
	Let $b = \sum_{i=1}^n a_i \nu_i$. Then
	\[ \lvert b \rvert \leq \sum_{i=1}^n \lvert a_i \rvert,\]
	since $\nu_1, \ldots, \nu_n$ all have real part in the interval $[0, 1)$.
	In particular,
	\[ \lvert b \rvert \leq c_5 \Delta^{c_4}.\]
	The tuple $(x_1, \ldots, x_n, y)$ thus gives rise to the point
	\[ (\tau_1, \ldots, \tau_n, \tau, a_1, \ldots, a_n, b) \in Z,\]
	which is quadratic in the $\tau_i, \tau$ coordinates and integral in the $a_i, b$ coordinates. Further, the $2$-height of this point is $\leq c_6 \Delta^{c_7}$ by Proposition~\ref{prop:hts} and the above bound on the multiplicative height.
	
	Every Galois conjugate $(x_1', \ldots, x_n', y')$ of $(x_1, \ldots, x_n, y)$ over $K$ satisfies the multiplicative relation
	\[ \prod_{i=1}^n (x_i' - y')^{a_i} =1,\]
	where $a_1, \ldots, a_n$ are the same integers as before. The conjugate $(x_1', \ldots, x_n', y')$ thus gives rise, in the same way as $(x_1, \ldots, x_n, y)$ did, to a point
	\[ (\tau_1', \ldots, \tau_n', \tau', a_1, \ldots, a_n, b') \in Z,\]
	where the $\tau_i', \tau'$ are quadratic and of multiplicative height $\leq 4 \Delta/3$, the $a_i$ are the same integers as before, and $b'$ is an integer such that $\lvert b' \rvert \leq c_5 \Delta^{c_4}$. Note that $b'$ is not necessarily the same as $b$. In particular, the point $(\tau_1', \ldots, \tau_n', \tau', a_1, \ldots, a_n, b')$ also has $2$-height $\leq c_6 \Delta^{c_7}$. Further, the corresponding points of $Z$ arising from distinct $K$-conjugates of $(x_1, \ldots, x_n, y)$ are always distinct in the $\mathfrak{F}_j^{n+1}$ coordinates. By Proposition~\ref{prop:galbd} with $\epsilon = 1/4$, there are at least $c_8 \Delta^{1/4}$ distinct $K$-conjugates of $(x_1, \ldots, x_n, y)$, each of which gives rise to a distinct point of $Z$ in the above way. 
	
	View $Y$ as a definable family of sets fibred over the $(r_1, \ldots, r_n)$ coordinates. Each of the points 
	\[ (\tau_1', \ldots, \tau_n', \tau', a_1, \ldots, a_n, b') \in Z\]
	described above is the projection of a point
	\[ (\tau_1', \ldots, \tau_n', \tau', \tau_1', \ldots, \tau_n', \tau', \nu_1', \ldots, \nu_n', a_1, \ldots, a_n, b') \in Y.\]
	Note that the $Y$-points arising in this way from distinct conjugates over $K$ of $(x_1, \ldots, x_n, y)$ are distinct in their $(\tau_1', \ldots, \tau_n', \tau')$ coordinates.
	
	These points $(\tau_1', \ldots, \tau_n', \tau', \tau_1', \ldots, \tau_n', \tau', \nu_1', \ldots, \nu_n', b')$ all lie on the fibre $Y_{(a_1, \ldots, a_n)}$ of $Y$ over $(a_1, \ldots, a_n)$. Let 
	\[\pi_1 \colon Y_{(a_1, \ldots, a_n)} \to \mathfrak{F}_j^{n+1} \times \R \mbox{ and }\pi_2 \colon Y_{(a_1, \ldots, a_n)} \to \mathfrak{F}_j^{n+1} \times \mathfrak{F}_e^{n}\] 
	be the projection maps sending
	\[ (z_1, \ldots, z_n, z, w_1, \ldots, w_n, w, u_1, \ldots, u_n, s) \mapsto (z_1, \ldots, z_n, z, s)\]
	and
	\[ (z_1, \ldots, z_n, z, w_1, \ldots, w_n, w, u_1, \ldots, u_n, s) \mapsto (w_1, \ldots, w_n, w, u_1, \ldots, u_n)\]
	respectively. Observe that $\pi_2$ is injective. Let $\Sigma \subset Y_{(a_1, \ldots, a_n)}$ be the set consisting of all the points arising in the way described above from the $K$-conjugates of $(x_1, \ldots, x_n, y)$. Then $\pi_1(\Sigma)$ contains only algebraic points of degree at most $2$ and which have $2$-height $\leq c_6 \Delta^{c_7}$. Also, $\# \pi_2(\Sigma) > c_8 \Delta^{1/4}$.
	
	Now let $C_{HP}>0$ be the constant given by Theorem~\ref{thm:count} applied to $Y$ with $k=2$ and $\epsilon = (8 c_7)^{-1}$. Note that $c_7$ depends only on $n$, which is fixed, and $C_{HP}$ depends only on $Y, k, \epsilon$, which are all fixed. In particular, $C_{HP}$ is independent of $(x_1, \ldots, x_n, y)$ and $a_1,\ldots, a_n$. Let $T= c_6 \Delta^{c_7}$. Then $\pi_1(\Sigma)$ contains only algebraic points of degree $\leq 2$ and $2$-height $\leq T$ and
	\[ \# \pi_2(\Sigma) > c_8 \Delta^{1/4}.\]
	In particular, if $\Delta > (c_6 C_{HP} / c_8)^8$, then
	\[ \# \pi_2(\Sigma) > C_{HP} T^{\epsilon}\]
	(here we assume without loss of generality that $c_6, c_7 \geq 1$, so $c_6^\epsilon \leq c_6$).
	
	Suppose then that $\Delta > (c_6 C_{HP} / c_8)^8$. Theorem~\ref{thm:count} implies that there exists a continuous, definable function $\beta \colon [0, 1] \to Y_{(a_1, \ldots, a_n)}$ with the following properties:
	\begin{enumerate}
		\item The composition $\pi_1 \circ \beta \colon [0, 1] \to \mathfrak{F}_j^{n+1} \times \R$ is semialgebraic and its restriction to $(0, 1)$ is real analytic.
		\item The composition $\pi_2 \circ \beta \colon [0, 1] \to \mathfrak{F}_j^{n+1} \times \mathfrak{F}_e^n$ is non-constant.
		\item $\pi_2(\beta(0)) \in \pi_2(\Sigma)$.
		\item The restriction of $\beta$ to $(0, 1)$ is real analytic.
	\end{enumerate}
	Note that, by the construction of the set $Y$, property (2) implies that $\pi_1 \circ \beta$ composed with projection to the $\mathfrak{F}_j^{n+1}$ coordinates is non-constant. Since $\pi_2$ is injective, we have that $\beta(0) \in \Sigma$, i.e.~$\beta(0)$ is a point of $Y_{(a_1, \ldots, a_n)}$ arising from some $K$-conjugate $(x_1', \ldots, x_n', y')$ of $(x_1, \ldots, x_n, y)$.
	
	Let 
	\[ V_{(a_1, \ldots, a_n)} = \left\{ (z_1, \ldots, z_n, z, w) \in \C^{n+1} \times \C^\times : \prod_{i=1}^n (z_i - z)^{a_i} = w\right\}.\]
	Let $\mathcal{V} = \pi^{-1}(V_{(a_1, \ldots, a_n)})$, where $\pi = (j, \ldots, j, e)$ is the map defined in Section~\ref{subsec:Ax}. Observe that $(\pi_1 \circ \beta)([0, 1])$ is a connected, positive-dimensional semialgebraic set contained in $\mathcal{V}$. 
	
	Note that $V_{(a_1, \ldots, a_n)}$ is an algebraic subvariety of $\C^{n+1} \times \C^\times$. Hence we may apply \cite[Proposition~6.2]{Pila11} to $\mathcal{V}$ and thereby obtain a complex algebraic component $W$ of $\mathcal{V}$ such that $\pi_1(\beta(0)) \in W$. Note that $W$ is, by definition, positive-dimensional. Enlarging $W$ as necessary, we may assume that $W$ is a maximal complex algebraic component of $\mathcal{V}$.
	
	The Ax--Lindemann result of Theorem~\ref{thm:AxL} thus implies that $W$ is a weakly special subvariety of $\h^{n+1} \times \C$. So $W = W_1 \times W_2$, where $W_1$ is a weakly special subvariety of $\h^{n+1}$ and $W_2$ is a weakly special subvariety for $e$ of $\C$. Since $W \subset \mathcal{V}$ and the projection $\mathcal{V} \to \h^{n+1}$ has discrete fibres, the weakly special subvariety $W_2$ must just be a point. Hence, $W_2$ is equal to the projection of $\pi_1(\beta(0))$, which is $\{b'\}$ for some $b' \in \Z$. Moreover, $W_1$ must be positive-dimensional. 
	
	Also, $W_1$ contains the preimage in $\mathfrak{F}_j^{n+1}$ of the $K$-conjugate $(x_1', \ldots, x_n', y')$ of $(x_1, \ldots, x_n, y)$, since $\pi_1(\beta(0)) \in W$. Finally, note that
	\begin{align}\label{eq:predep}
		W_1 \subset \left\{ (z_1, \ldots, z_n, z) : \prod_{i=1}^n (j(z_i) - j(z))^{a_i} =1\right\},
	\end{align}
	since $W_2 = \{b'\}$ and $b' \in \Z$. 
	
	The image $j(W_1)$ is therefore a positive-dimensional weakly special subvariety of $\C^{n+1}$ which contains $(x_1', \ldots, x_n', y')$. In particular, $j(W_1)$ is in fact a special subvariety of $\C^{n+1}$, because $j(W_1)$ contains the special point $(x_1', \ldots, x_n', y')$. 
	
	Let $l = \dim j(W_1)$. A special subvariety of $\C^{n+1}$ of dimension $l$ is equal, up to reordering coordinates, to some Cartesian product $M_1 \times \ldots \times M_l$, where $M_1 \subset \C^{m_1}, \ldots, M_l \subset \C^{m_l}$ are one-dimensional special subvarieties and $m_1, \ldots, m_l \in \Z_{>0}$, see e.g.~\cite[pp.~33--34]{Pila22}. Hence, after reordering only the first $n$ coordinates of $j(W_1)$, we have that
	\[ j(W_1) = M_1 \times \ldots \times M_l\]
	for some one-dimensional special subvarieties $M_1 \subset \C^{m_1}, \ldots, M_l \subset \C^{m_l}$. By Proposition~\ref{prop:jset}, for each $i \in \{1, \ldots, l\}$, there exist $j$-maps $f_{i, 1}, \ldots, f_{i, m_i}$, which are not all constant, such that
	\[ M_i = \{ (f_{i, 1}(z_i), \ldots, f_{i, m_i}(z_i)) : z_i \in \h\}.\]
	
	By \eqref{eq:predep}, we have that
	\begin{align}\label{eq:reindexdep}
		\prod_{i=1}^l \prod_{\substack{k=1\\ (i, k) \neq (l, m_l)}}^{m_i} (f_{i, k}(z_i) - f_{l, m_l}(z_l))^{a_{i, k}} = 1
	\end{align}
	for all $(z_1, \ldots, z_l) \in \h^l$, where the $a_{i, k}$ are the appropriate re-indexing of the $a_i$ in \eqref{eq:predep}. For $i \in \{1, \ldots, l-1\}$, let $\tau_i \in \h$ be such that
	\[(f_{i, 1}(\tau_i), \ldots, f_{i, m_i}(\tau_i)) = \pi_{M_i}((x_1', \ldots, x_n', y')),\]
	where $\pi_{M_i} \colon \C^{n+1} \to \C^{m_i}$ denotes the projection map onto the coordinates corresponding to $M_i$. Such a $\tau_i$ exists since $(x_1', \ldots, x_n', y') \in j(W_1)$. Let
	\[ M_i(\tau_i) = \{ (f_{i, 1}(\tau_i), \ldots, f_{i, m_i}(\tau_i))\}.\]
	
	Now let
	\[ M = \prod_{i=1}^{l-1} M_i(\tau_i) \times M_l.\]
	Then $M$ is a one-dimensional special subvariety of $\C^{n+1}$. By Proposition~\ref{prop:jset}, there exist $j$-maps $f_1, \ldots, f_n, f$ such that
	\[ M = \{(f_1(z), \ldots, f_n(z), f(z)) : z \in \h\}.\]
	Note that at least one of $f_1, \ldots, f_n, f$ is non-constant. By \eqref{eq:reindexdep}, we have that
	\begin{align*}\label{eq:mainpfend}
		\prod_{i=1}^n (f_i(z) - f(z))^{a_i} =1 \mbox{ for all } z \in \h.
	\end{align*}
	After reordering only the first $n$ coordinates of $M$, we may assume that $(x_1', \ldots, x_n', y') \in M$. Since $x_1', \ldots, x_n', y'$ are pairwise distinct, the $j$-maps $f_1, \ldots, f_n, f$ are therefore pairwise distinct. Therefore, $M$ is a multiplicative special curve.
	
	Thus, $M$ is one of the finitely many multiplicative special curves in $\C^{n+1}$ given by Theorem~\ref{thm:MSC}. In particular, $M$ is defined over $K$. Thus, $(x_1, \ldots, x_n, y) \in M$, since $M$ contains the $K$-conjugate $(x_1', \ldots, x_n', y')$ of $(x_1, \ldots, x_n, y)$.
	
	We have therefore shown that, for $(x_1, \ldots, x_n, y)$ an $(n+1)$-tuple of pairwise distinct singular moduli $x_1, \ldots, x_n, y$ of complexity $\Delta$ such that 
	\[ \prod_{i=1}^{n} (x_i - y)^{a_i} =1\]
	for some $a_1, \ldots, a_n \in \Z \setminus \{0\}$, if $\Delta > (c_6 C_{HP} / c_8)^8$, then $(x_1, \ldots, x_n, y)$ belongs to one of the finitely many multiplicative special curves in $\C^{n+1}$. Hence, the complexity of every such $(n+1)$-tuple which does not lie on a multiplicative special curve in $\C^{n+1}$ is $\leq (c_6 C_{HP} / c_8)^8$. In particular, there are only finitely many such $(n+1)$-tuples. This completes the proof of Theorem~\ref{thm:main}. Corollary~\ref{cor:small} follows immediately.

\section{The Zilber--Pink connection}\label{sec:ZP}

Let $m, n \in \Z_{>0}$. Let
\[ X_{m, n} = \C^m \times (\C^\times)^n.\]
Recall the definition of a special subvariety of $X_{m, n}$ from Definition~\ref{def:specials}.

\begin{definition}\label{def:atyp}
	Let $V \subset X_{m, n}$ be a subvariety. A subvariety $W \subset V$ is called an atypical component of $V$ in $X_{m, n}$ if there exists a special subvariety $T \subset X_{m, n}$ such that $W$ is an irreducible component of $V \cap T$ and
	\[ \dim W > \dim V + \dim T - \dim X_{m,n}.\]
	An atypical component $W$ of $V$ in $X_{m, n}$ is a maximal atypical component of $V$ in $X_{m, n}$ if there does not exist any atypical component $W'$ of $V$ in $X_{m, n}$ such that $W \subsetneq W'$.
\end{definition}

The Zilber--Pink conjecture was formulated independently in different contexts by Zilber \cite{Zilber02}, Pink \cite{Pink05}, and Bombieri, Masser, and Zannier \cite{BombieriMasserZannier07}. The conjecture is wide open; see \cite[Part~IV]{Pila22} for more details. In our context, the Zilber--Pink conjecture is the following statement.

\begin{conj}[Zilber--Pink conjecture]\label{conj:ZP}
	Let $V \subset X_{m, n}$ be a subvariety. Then there are only finitely many maximal atypical components of $V$ in $X_{m, n}$.
\end{conj}

In the remainder of this section, we show that, in light of Theorem~\ref{thm:MSC}, Theorem~\ref{thm:main} would follow from Conjecture~\ref{conj:ZP}. For $n \in \Z_{>0}$, we define $V_n \subset X_{n+1, n}$ by
\[ V_n =\{(w_1, \ldots, w_n, w, t_1, \ldots, t_n) \in X_{n+1, n} : t_i = w_i - w \mbox{ for } i =1, \ldots, n\}.\]
Note that $\dim X_{n+1, n} = 2n+1$ and $\dim V_n = n+1$.

\begin{prop}\label{prop:atyp}
	Let $n \in \Z_{>0}$. Suppose that $x_1, \ldots, x_n, y$ are singular moduli such that $x_i \neq y$ for $i \in \{1, \ldots, n\}$ and
	\[\prod_{i=1}^n (x_i - y)^{a_i} = 1\]
	for some $a_1, \ldots, a_n \in \Z$ which are not all zero. Then
	\[\{(x_1, \ldots, x_n, y, x_1 - y, \ldots, x_n - y) \}\]
	is an atypical component of $V_n$ in $X_{n+1, n}$.
\end{prop}

\begin{proof}
	Let
	\[ \sigma = (x_1, \ldots, x_n, y, x_1 - y, \ldots, x_n - y).\]
	Observe that $\sigma \in V_n$. Since $x_1, \ldots, x_n, y$ are singular moduli, the set
	\[\{(x_1, \ldots, x_n, y)\}\]
	is a special subvariety of $\C^{n+1}$ of dimension $0$. Write $M$ for this special subvariety. Since $x_i \neq y$ and
	\[\prod_{i=1}^n (x_i - y)^{a_i} = 1,\]
	the point 
	\[(x_1 - y, \ldots, x_n - y)\]
	is contained in a special subvariety $T \subset (\C^\times)^n$ of dimension at most $n-1$.
	Hence, $M \times T$ is a special subvariety of $X_{n+1, n}$ of dimension $\leq n-1$. Thus,
	\[ \dim V_n + \dim (M \times T) - \dim X_{n+1, n} \leq (n+1) + (n-1) - (2n+1)  < 0.\]
	Thus, $\{\sigma\} \subset V_n \cap (M \times T)$ is an atypical component of $V_n$ in $X_{n+1, n}$.
\end{proof}

\begin{prop}\label{prop:minatyp}
	Let $n \in \Z_{>0}$. Suppose that $x_1, \ldots, x_n, y$ are pairwise distinct singular moduli such that the set $\{x_1-y, \ldots, x_n - y\}$ is minimally multiplicatively dependent. Then either $(x_1, \ldots, x_n, y)$ lies on a multiplicative special curve in $\C^{n+1}$ or 
	\[ \{ (x_1, \ldots, x_n, y, x_1 - y, \ldots, x_n - y)\}\]
	is a maximal atypical component of $V_{n}$ in $X_{n+1, n}$.
\end{prop}

\begin{proof}
	Suppose that $x_1, \ldots, x_n, y$ are pairwise distinct singular moduli such that the set $\{x_1-y, \ldots, x_n - y\}$ is minimally multiplicatively dependent. Let
	\[ \sigma = (x_1, \ldots, x_n, y, x_1 - y, \ldots, x_n - y).\]
	Then, by Proposition~\ref{prop:atyp},	$\{\sigma\}$ is an atypical component of $V_{n}$ in $X_{n+1, n}$.
	
	Suppose then that $\{\sigma\}$ is not a maximal atypical component of $V_{n}$ in $X_{n+1, n}$. Then there exist special subvarieties $M \subset \C^{n+1}$ and $T \subset (\C^\times)^{n}$ and an irreducible component $W \subset V_n \cap (M \times T)$ such that $\dim W > 0$, $\sigma \in W$, and
	\[ \dim W > \dim V_n + \dim (M \times T) - \dim X_{n+1, n} = \dim M + \dim T - n.\]
	
	If $T = \G^n$, then
	\[ V_n \cap (M \times T) = \{(w_1, \ldots, w_n, w, w_1 - w, \ldots, w_n - w) : (w_1, \ldots, w_n) \in M\}\]
	and hence any component of this intersection clearly has dimension $\leq \dim M$ and so cannot be an atypical component. Similarly, if $M = Y(1)^{n+1}$, then any component of the intersection $V_n \cap (M \times T)$ has dimension $\leq \dim T + 1$ and hence is not an atypical component. We may thus assume that $M, T$ are both proper subvarieties.
	
	If $T$ was defined by two independent multiplicative conditions, then two independent multiplicative relations would hold on the set
	\[ \{x_1 - y, \ldots, x_n - y\},\]
	and thus some proper subset would be multiplicatively dependent, a contradiction. So $T$ must be defined by one independent multiplicative condition. 
	
	Hence, for $M \times T$ to intersect $V_n$ atypically, one must have that
	\[(\alpha_1 - \beta, \ldots, \alpha_n - \beta) \in T\]
	for every $(\alpha_1, \ldots, \alpha_n, \beta) \in M$. Thus, by Proposition~\ref{prop:jset}, if $M_0 \subset M$ is a special subvariety of $\C^{n+1}$ such that $\dim M_0 = 1$ and no two coordinates of $M_0$ are identically equal, then $M_0$ is a multiplicative special curve in $\C^{n+1}$.
	
	Suppose that $\dim M > 1$. Since $(x_1, \ldots, x_n, y) \in M$ and $x_1, \ldots, x_n, y$ are pairwise distinct, the locus in $M$ where some two coordinates are equal is a Zariski-closed proper subset of $M$. Thus, by Proposition~\ref{prop:specialsdense}, $M$ must contain infinitely many multiplicative special curves in $\C^{n+1}$. However, there are only finitely many multiplicative special curves in $\C^{n+1}$ by Theorem~\ref{thm:MSC}. So we must have that $\dim M = 1$, and so $M$ itself is a multiplicative special curve in $\C^{n+1}$. Since $(x_1, \ldots, x_n, y) \in M$, the proof is complete.
\end{proof}

\begin{prop}
	Assume Conjecture~\ref{conj:ZP}. Then Theorem~\ref{thm:MSC} implies Theorem~\ref{thm:main}.
\end{prop}

\begin{proof}
	Fix $n \in \Z_{>0}$. We will show that there exists a constant $C>0$, depending only on $n$, with the following property. Suppose that $(x_1, \ldots, x_n, y)$ is an $(n+1)$-tuple of pairwise distinct singular moduli $x_1, \ldots, x_n, y$ such that 
	\[ \prod_{i=1}^n (x_i - y)^{a_i} = 1\]
	for some $a_1, \ldots, a_n \in \Z \setminus \{0\}$. Then either $\lvert \Delta(y) \rvert \leq C$ or $(x_1, \ldots, x_n, y)$ lies on a multiplicative special curve in $\C^{n+1}$. By Theorem~\ref{thm:fixedy}, this suffices to prove Theorem~\ref{thm:main}. In what follows, we let $c_1, c_2, \ldots$ denote positive constants which depend only on $n$. 
	
	Suppose that $(x_1, \ldots, x_n, y)$ is an $(n+1)$-tuple of pairwise distinct singular moduli $x_1, \ldots, x_n, y$ such that 
	\[ \prod_{i=1}^n (x_i - y)^{a_i} = 1\]
	for some $a_1, \ldots, a_n \in \Z \setminus \{0\}$. For every $k \in \{1, \ldots, n\}$, there exists, by \cite[Lemma~5.9]{Fowler21}, a set $I_k \subset \{1, \ldots, n\}$ such that $k \in I_k$ and the set
	\[ \{x_i - y : i \in I_k\}\]
	is minimally multiplicatively dependent. Let $m_k = \# I_k$. Let $i_{k, 1}, \ldots, i_{k, m_k} \in \{1, \ldots, n\}$ be pairwise distinct and such that $I_k = \{i_{k, 1}, \ldots, i_{k, m_k}\}$. 
	
	Conjecture~\ref{conj:ZP} implies that, for every $k \in \{1, \ldots, n\}$, there are only finitely many maximal atypical components of $V_{m_k}$ in $X_{m_k+1, m_k}$. Therefore, there exists a constant $c_1$ with the property that if there exists $k$ such that
	\[ \{(x_{i_{k, 1}}, \ldots, x_{i_{k, m_k}}, y, x_{i_{k, 1}}- y, \ldots, x_{i_{k, m_k}} - y)\}\]
	is a maximal atypical component of $V_{m_k}$ in $X_{m_k+1, m_k}$, then $\lvert \Delta(y) \rvert \leq c_1$. We may therefore assume that, for every $k \in \{1, \ldots, n\}$, the point
	\[ \{(x_{i_{k, 1}}, \ldots, x_{i_{k, m_k}}, y, x_{i_{k, 1}}- y, \ldots, x_{i_{k, m_k}} - y)\}\]
	is not a maximal atypical component of $V_{m_k}$ in $X_{m_k+1, m_k}$. 
	
	Proposition~\ref{prop:minatyp} then implies that, for every $k \in \{1, \ldots, n\}$, there exists a multiplicative special curve $M_k \subset \C^{m_k+1}$ such that
	\[(x_{i_{k, 1}}, \ldots, x_{i_{k, m_k}}, y) \in M_k.\] 
	By Theorem~\ref{thm:MSCshape}, there exist $j$-maps $f_{k, 1}, \ldots, f_{k, m_k}$ such that
	\[ M_k = \{ (f_{k, 1}(z), \ldots, f_{k, m_k}(z), j(z)) : z \in \h\}.\]
	Theorem~\ref{thm:MSC} implies that there exists a constant $c_2$ such that:
	\begin{enumerate}
		\item if $f_{k, r}$ is a constant $j$-map whose value is a singular modulus of discriminant $\Delta$, then $\lvert \Delta \rvert \leq c_2$,
		\item if $f_{k, r}$ is a non-constant $j$-map, then there exist $N \in \Z_{>0}$ and $g \in C(N)$ such that $N \leq c_2$ and $f_{k, r}(z) = j(g z)$.
	\end{enumerate}
	If $f_{k, r}$ is a constant $j$-map, then $f_{k, r} = x_{i_{k, r}}$ and hence $\lvert \Delta(x_{i_{k, r}}) \rvert \leq c_2$. 
	
	Let the constant $c_3$ be such that $\lvert \Delta(x) \rvert \leq c_3$ for every singular modulus $x$ belonging to the set
	\begin{align*}
		\{ z \in \C :\, &\mbox{there exist } M \in \Z_{>0} \mbox{ and a singular modulus } w\\ 
		&\mbox{such that} \max\{M,  \lvert \Delta(w) \rvert \} \leq c_2 \mbox{ and } \Phi_M(w, z) = 0 \}.
	\end{align*}
	Such a constant exists since this set is clearly finite. For $N \in \Z_{>0}$, denote by $\mathbb{V}(\Phi_N)$ the vanishing locus of the polynomial $\Phi_N(X, Y)$ in $\C^2$.
	The sets $\mathbb{V}(\Phi_N)$ are pairwise distinct, irreducible plane curves. So if $M \neq N$, then the intersection  $ \mathbb{V}(\Phi_M) \cap \mathbb{V}(\Phi_N)$
	is a finite set. There thus exists a constant $c_4$ with the property that if $z, w$ are singular moduli such that
	\[ (z, w) \in \bigcup_{\substack{M, N \in \Z_{>0}\\ M \neq N \mbox{ and } M, N \leq c_2}} ( \mathbb{V}(\Phi_M) \cap \mathbb{V}(\Phi_N)),\]
	then $\max \{ \lvert \Delta(w) \rvert, \lvert \Delta(y) \rvert\} \leq c_4$.
	
	We now claim that, for every $r \in \{1, \ldots, n\}$, either
	\[ \{f_{k, s} : i_{k, s} = r\} = \{x_r\}\]
	or there exists $N \leq c_2$ such that
	\[\{f_{k, s}(z) : i_{k, s} = r\} \subset \{j(g z) : g \in C(N)\}.\]
	It suffices to show that if $r \in \{1, \ldots, n\}$ is such that $r = i_{k, l} = i_{k', l'}$, then either $f_{k, l} = f_{k', l'} = x_r$ or there exist $N \leq c_2$ and $g, h \in C(N)$ such that $f_{k, l}(z) = j(g z)$ and $f_{k', l'}(z) = j(h z)$.   Relabelling as necessary, we may assume that $r = i_{k, 1} = i_{k', 1}$.
	
	If $f_{k, 1}, f_{k', 1}$ are both constant, then $f_{k, 1} = f_{k', 1} = x_r$. Next, consider the case that the $j$-map $f_{k, 1}$ is constant and the $j$-map $f_{k', 1}$ is non-constant. So $f_{k, 1} = x_{r}$ and hence $\lvert \Delta(x_r) \rvert \leq c_2$. Also, there exist $N \in \Z_{>0}$ and $g \in C(N)$ such that $N \leq c_2$ and $f_{k', 1}(z) = j(g z)$. Since $(x_{i_{k', 1}}, \ldots, x_{i_{k', m_{k'}}}, y) \in M_{k'}$, there exists some $z_0 \in \h$ such that $(j(g z_0), j(z_0)) = (x_{r}, y)$. In particular, $\Phi_N(x_{r}, y) = 0$. Therefore, $\lvert \Delta(y) \rvert \leq c_3$.
	
	Finally, suppose that the $j$-maps $f_{k, 1}, f_{k', 1}$ are both non-constant. Then there exist $M, N \in \Z_{>0}$ and $g \in C(M), h \in C(N)$ such that $f_{k, 1}(z) = j(g z)$ and $f_{k', 1}(z) = j(h z)$. Since
	\[(x_{i_{k, 1}}, \ldots, x_{i_{k, m_{k'}}}, y) \in M_{k} \mbox{ and } (x_{i_{k', 1}}, \ldots, x_{i_{k', m_{k'}}}, y) \in M_{k'},\] 
	we have that
	\[ (x_r, y) \in \{ (j(g z), j(z)) : z \in \h\} \cap \{ (j(h w), j(w)) : w \in \h\}. \]
	Hence, $(x_r, y) \in  \mathbb{V}(\Phi_M) \cap \mathbb{V}(\Phi_N)$, and thus $\lvert \Delta(y) \rvert \leq c_4$ if $M \neq N$. The claim therefore holds if $\lvert \Delta(y) \rvert \geq \max \{c_3, c_4\}$.
	
	Suppose then that $\lvert \Delta(y) \rvert \geq \max \{c_3, c_4\}$. The proved claim and the properties of multiplicative special curves given in Theorem~\ref{thm:MSCshape} together imply that there exists a partition $P_0, \ldots, P_l$ of $\{1, \ldots, n\}$, where $l \geq 1$, with the following properties:
	\begin{enumerate}
		\item The $j$-map $f_{k, r}$ is constant if and only if $i_{k, r}\in P_0$, in which case $f_{k, r} = x_{i_{k, r}}$.
		\item For $u \in \{1, \ldots, l\}$ and $k \in \{1, \ldots, n\}$, if $P_u \cap I_k \neq \emptyset$, then $P_u \subset I_k$.
		\item For $u \in \{1, \ldots, l\}$, there exists $N_u \in \Z_{>0}$ such that writing $C(N_u) = \{g_1, \ldots, g_{s}\}$ with $g_1, \ldots, g_s$ pairwise distinct, if $k \in \{1, \ldots, n\}$ is such that $P_u \subset I_k$, then 
		\begin{align*} &\{f_{k, v}(z) :  v \in \{1, \ldots, m_k\} \mbox{ such that } i_{k, v} \in P_u \}\\ 
			&= \{j(g_1 z), \ldots, j(g_s z)\}.
		\end{align*}
	\end{enumerate}
	
	We will now define $j$-maps $f_1, \ldots, f_n$. For $m \in P_0$, let $f_m = x_m$. For $m \in I_1 \setminus P_0$, let $f_m(z) = f_{1, v}(z)$, where $v$ is the unique integer such that $i_{1, v} = m$. Now let $z_0 \in \h$ be such that $j(z_0) = y$ and $f_m(z_0) = x_m$ for every $m \in I_1 \cup P_0$. Such a $z_0$ exists by the definition of $M_1$ and $P_0$.
	
	Let $r \geq 1$ be such that $P_r \not \subset I_1$. Write $C(N_r) = \{g_1, \ldots, g_s\}$ with $g_1, \ldots, g_s$ pairwise distinct. Note that there exists $w_0 \in \h$ such that
	\[ j(w_0) = y \mbox{ and } \{j(g_1 w_0), \ldots, j(g_s w_0)\} = \{x_i : i \in P_r\}. \]
	This follows from property (3) of the constructed partition by considering the multiplicative special curve $M_k$ for some $k$ such that $P_r \subset I_k$. In particular, $y= j(z_0) = j(w_0)$ and so there exists $\gamma \in \SL$ such that $\gamma z_0 = w_0$. Note that the functions $j(g_1 \gamma z), \ldots, j(g_s \gamma z)$ are just a permutation of the functions $j(g_1 z), \ldots, j(g_s z)$. Hence, we also have that
	\[ \{j(g_1 z_0), \ldots, j(g_s z_0)\} = \{x_i : i \in P_r\}.\]
	Since the $x_i$ are pairwise distinct, we may then, for $m \in P_r$, define $f_m(z) = j(g_v z)$, where $v \in \{1, \ldots, s\}$ is the unique integer $v$ such that $j(g_v z_0) = x_m$. 
	
	Since the $P_i$ partition $\{1, \ldots, n\}$, we may define in this way $j$-maps $f_1, \ldots, f_n$. By construction, we have that
	\[ (x_1, \ldots, x_n, y) = (f_1(z_0), \ldots, f_n(z_0), j(z_0)).\]
	In particular, the functions $f_1, \ldots, f_n, j$ are pairwise distinct. For $k \in \{1, \ldots, n\}$, observe that
	\[ \{f_{i}(z) : i \in I_k\} = \{ f_{k, r}(z) : r \in \{1, \ldots, m_k\} \} .\]
	Since $M_k$ is a multiplicative special curve and this property is preserved under permutation of the first $m_k$ coordinates of $M_k$, there therefore exist $a_{k, i} \in \Z \setminus \{0\}$ such that
	\[\prod_{i \in I_k} (f_i(z) - j(z))^{a_{k, i}} = 1\] 
	for every $k \in \{1, \ldots, n\}$ and $z \in \h$. Since $k \in I_k$ for every $k$, Lemma~\ref{lem:vect} implies that there exist $b_1, \ldots, b_n \in \Z \setminus \{0\}$ such that
	\[\prod_{i =1}^n (f_i(z) - j(z))^{b_i} = 1\]
	for all $z \in \h$. Thus, the set
	\[ \{(f_1(z), \ldots, f_n(z), j(z)) : z \in \h\}\]
	is a multiplicative special curve in $\C^{n+1}$ which contains $(x_1, \ldots, x_n, y)$.
	
	We have thus shown that if $(x_1, \ldots, x_n, y)$ is an $(n+1)$-tuple of pairwise distinct singular moduli $x_1, \ldots, x_n, y$ such that 
	\[ \prod_{i=1}^n (x_i - y)^{a_i} = 1\]
	for some $a_1, \ldots, a_n \in \Z \setminus \{0\}$, then either $\lvert \Delta(y) \rvert \leq \max \{c_1, c_3, c_4\}$ or $(x_1, \ldots, x_n, y)$ lies on a multiplicative special curve in $\C^{n+1}$. Theorem~\ref{thm:main} thus follows from Theorem~\ref{thm:fixedy}.
\end{proof}

\bibliographystyle{amsplain}

\end{document}